\documentclass[11pt, letterpaper]{amsart}

\newcommand{\la}{\langle}
\newcommand{\ra}{\rangle}
\renewcommand{\Re}{\operatorname{Re}}
\renewcommand{\Im}{\operatorname{Im}}

\usepackage{amssymb}
\usepackage{amsthm}
\usepackage{amsxtra}
\usepackage{graphicx}
\usepackage{textcomp, verbatim, color, mathrsfs, mathabx, braket}

\newtheorem{theorem}{Theorem}

\newtheorem{proposition}[theorem]{Proposition}
\newtheorem{lemma}[theorem]{Lemma}
\newtheorem{corollary}[theorem]{Corollary}

\newtheorem{conjecture}[theorem]{Conjecture}

\theoremstyle{remark}
\newtheorem{remark}[theorem]{Remark}

\numberwithin{equation}{section}

\numberwithin{theorem}{section}

\numberwithin{table}{section}

\numberwithin{figure}{section}

\ifx\pdfoutput\undefined
  \DeclareGraphicsExtensions{.pstex, .eps}
\else
  \ifx\pdfoutput\relax
    \DeclareGraphicsExtensions{.pstex, .eps}
  \else
    \ifnum\pdfoutput>0
      \DeclareGraphicsExtensions{.pdf}
    \else
      \DeclareGraphicsExtensions{.pstex, .eps}
    \fi
  \fi
\fi

\title[Scattering for NLS with a Potential]{Scattering for a Nonlinear Schr\"odinger Equation with a Potential}
\author{ Younghun Hong }
\address{Department of Mathematics \newline\indent The Uxniversity of Texas at Austin}
\email{yhong@math.utexas.edu}

\begin{document}

\begin{abstract}
We consider a 3d cubic focusing nonlinear Schr\"odinger equation with a potential
$$i\partial_t u+\Delta u-Vu+|u|^2u=0,$$
where $V$ is a real-valued short-range potential having a small negative part. We find criteria for global well-posedness analogous to the homogeneous case $V=0$ \cite{HR, DHR}. Moreover, by the concentration-compactness approach, we prove that if $V$ is repulsive, such global solutions scatter.
\end{abstract}

\date{\today}
\linespread{1.2}
\maketitle

\section{Introduction}

\subsection{Setup of the problem}
We consider a 3d cubic focusing nonlinear Schr\"odinger equation with a potential 
\begin{equation}\tag{$\textup{NLS}_V$}
i\partial_tu+\Delta u-Vu+|u|^2u=0,\ u(0)=u_0\in H^1,
\end{equation}
where $u=u(t,x)$ is a complex valued function on $\mathbb{R}\times\mathbb{R}^3$. We assume that $V=V(x)$ is a time independent real-valued short range potential having a small negative part. To be precise, we define the potential class $\mathcal{K}_0$ as the norm closure of bounded and compactly supported functions with respect to the \textit{global Kato norm}
$$\|V\|_{\mathcal{K}}:=\sup_{x\in\mathbb{R}^3} \int_{\mathbb{R}^3}\frac{|V(y)|}{|x-y|}dy,$$
and denote the negative part of $V$ by
$$V_-(x):=\min(V(x),0).$$ 
Throughout this paper, we assume that
\begin{equation}
V\in\mathcal{K}_0\cap L^{3/2}
\end{equation}
and
\begin{equation}\|V_-\|_{\mathcal{K}}<4\pi.
\end{equation}

By the assumptions $(1.1)$ and $(1.2)$, the Schr\"odinger operator $\mathcal{H}=-\Delta+V$ has no eigenvalues, and the solution to the linear Schr\"odinger equation
\begin{equation}\label{LS}
i\partial_t u+\Delta u-Vu=0,\ u(0)=u_0
\end{equation}
satisfies the dispersive estimate \cite{BG} and Strichartz estimates. As a consequence, a solution $u(t)$ to \eqref{LS} scatters in $L^2$ (see Lemma \ref{linear scattering}), in the sense that there exists $u_\pm\in L^2$ such that 
$$\lim_{t\to\pm\infty}\|u(t)-e^{it\Delta}u_\pm\|_{L^2}=0.$$

On the other hand, Holmer-Roudenko \cite{HR} and Duyckaerts-Holmer-Roudenko \cite{DHR} obtained the sharp criteria for global well-posedness and scattering for the homogeneous 3d cubic focusing nonlinear Schr\"odinger equation
\begin{equation}
i\partial_tu+\Delta u+|u|^2u=0,\ u(0)=u_0\in H^1
\end{equation}
in terms of conservation laws of the equation. Here, by homogeneity, we mean that $V=0$.\\

Motivated by the linear and nonlinear scattering results, it is of interest to investigate the effect of a potential perturbation on the scattering behavior of solutions to the nonlinear equation $(\textup{NLS}_V)$.

By the assumptions $(1.1)$ and $(1.2)$, the Cauchy problem for $(\textup{NLS}_V)$ is locally well-posed in $H^1$. Moreover, every $H^1$ solution obeys the mass conservation law,
$$M[u(t)]=\int_{\mathbb{R}^3} |u(t)|^2dx=M[u_0]$$
and the energy conservation law,
$$E[u(t)]=E_V[u(t)]=\frac{1}{2}\int_{\mathbb{R}^3}|\nabla u(t)|^2+V|u(t)|^2dx-\frac{1}{4}\int_{\mathbb{R}^3}|u(t)|^4dx=E[u_0].$$

The goal of this paper is to find criteria for global well-posedness and scattering in terms of the above two conserved quantities. Here, we say that a solution $u(t)$ to $(\textup{NLS}_V)$ \textit{scatters} in $H^1$ (both forward and backward in time) if there exist $\psi^\pm\in H^1$ such that
$$\lim_{t\to\pm\infty}\|u(t)-e^{-it\mathcal{H}}\psi^\pm\|_{H^1}=0.$$
Note that by the linear scattering (Lemma \ref{linear scattering}), if the solution $u(t)$ to $(\textup{NLS}_V)$ scatters in $H^1$, then there exist $\psi_0^\pm\in L^2$ such that 
$$\lim_{t\to\pm\infty}\|u(t)-e^{it\Delta}\psi_0^\pm\|_{L^2}=0.$$
In this way, we extend the works of Holmer-Roudenko \cite{HR} and Duyckaerts-Holmer-Roudenko \cite{DHR}.

\subsection{Criteria for global well-posedness}
In the first part of this paper, we find criteria for global well-posedness. As in the homogeneous case $(V=0)$, such criteria can be obtained from the variational problem that gives the sharp constant for the Gagliardo-Nirenberg inequality,
$$c_{GN}(V)=\sup_{u\in H^1,\ u\neq 0}\mathcal{W}_V(u),$$
where
$$\mathcal{W}_V(u)=\frac{\|u\|_{L^4}^4}{\|u\|_{L^2}\|\mathcal{H}^{1/2} u\|_{L^2}^3}.$$
When $V=0$, the sharp constant is attained at the ground state $Q$ solving the nonlinear elliptic equation
\begin{equation}\label{elliptic}
\Delta Q-Q+Q^3=0.
\end{equation}
The following proposition is analogous to the variational problem in the inhomogeneous case.

\begin{proposition}[Variational problem]\label{prop:SharpConstantGN}
Suppose that $V$ satisfies $(1.1)$ and $(1.2)$.\\
$(i)$ If $V_-=0$, then the sequence $\{Q(\cdot-n)\}_{n\in\mathbb{N}}$ maximizes $\mathcal{W}_V(u)$, where $Q$ is the ground state for the elliptic equation \eqref{elliptic}.\\
$(ii)$ If $V_-\neq0$, then there exists a maximizer $\mathcal{Q}\in H^1$ solving the elliptic equation 
\begin{equation}\label{eq:GroundStateEq}
(-\Delta+V)\mathcal{Q}+w_{\mathcal{Q}}^2 \mathcal{Q}-\mathcal{Q}^3=0,\ \omega_\mathcal{Q}=\tfrac{\|\mathcal{H}^{1/2}\mathcal{Q}\|_{L^2}}{\sqrt{3}\|\mathcal{Q}\|_{L^2}},
\end{equation}
Moreover, $\mathcal{Q}$ satisfies the Pohozhaev identities,
\begin{equation}
\|\mathcal{H}^{1/2} \mathcal{Q}\|_{L^2}^2=3\|\mathcal{Q}\|_{L^2}^2,\ \|\mathcal{Q}\|_{L^4}^4=4\|\mathcal{Q}\|_{L^2}^2.
\end{equation}
\end{proposition}

A related classical problem is to prove existence of ground states in the semi-classical setting \cite{FW, ABC}, which is, by change of variables, equivalent to 
\begin{equation}\label{semi-classical elliptic}
(-\Delta +V(\epsilon\cdot))u_\epsilon+\omega^2 u_\epsilon-|u_\epsilon|^2u_\epsilon=0
\end{equation}
for sufficiently small $\epsilon>0$, where $V$ is smooth and $\inf_{x\in\mathbb{R}^3}(\omega^2+V(\epsilon x))>0$. In \cite{ABC}, considering the equation \eqref{semi-classical elliptic} as a perturbation of 
$$-\Delta u+(\omega^2+V(0)) u-|u|^2u=0,$$
the authors found a ground state using a perturbation theorem in critical point theory. On the other hand, the ground state $\mathcal{Q}$ in Proposition \ref{prop:SharpConstantGN} $(ii)$ is obtained via the concentration-compactness approach based on profile decomposition \cite{G, HK}. From this, we obtain a ground state even when $V_-$ is not pointwise-bounded, while $V_-$ is still small in the global Kato norm.

\begin{remark}
The ground state $\mathcal{Q}$ is special in that it satisfies the ``exact" Pohozhaev identities. In general, solutions to $(\ref{eq:GroundStateEq})$ satisfy the Pohozhaev identities with extra terms (see Section 4.2). These exact identities will be crucially used to find criteria for global well-posedness.
\end{remark}

To state the main results, we need to introduce the following notation,
\begin{align*}
\mathcal{ME}&=\left\{\begin{aligned}
&M[Q]E_0[Q]&&\text{ if }V_-=0,\\
&M[\mathcal{Q}]E[\mathcal{Q}]&&\text{ if }V_-\neq0,
\end{aligned}
\right.\\
\alpha&=\left\{\begin{aligned}
&\|Q\|_{L^2}\|\nabla Q\|_{L^2}&&\text{ if }V_-=0,\\
&\|\mathcal{Q}\|_{L^2}\|\mathcal{H}^{1/2}\mathcal{Q}\|_{L^2}&&\text{ if }V_-\neq0,
\end{aligned}
\right.
\end{align*}
where $E_0[u]$ is the energy without a potential
$$E_0[u]=\frac{1}{2}\int_{\mathbb{R}^3}|\nabla u(x)|^2dx-\frac{1}{4}\int_{\mathbb{R}^3}|u(x)|^4dx.$$
Our first main theorem provides criteria for global well-posedness in terms of the mass-energy $\mathcal{ME}$ and a critical number $\alpha$.
\begin{theorem}[Upper-bound versus lower-bound dichotomy]\label{thm:ULDichotomy}
Suppose that $V$ satisfies $(1.1)$ and $(1.2)$. We assume that
$$M[u_0]E[u_0]<\mathcal{ME}.$$
Let $u(t)$ be the solution to $(\textup{NLS}_V)$ with initial data $u_0\in H^1$.\\
$(i)$ If
$$\|u_0\|_{L^2}\|\mathcal{H}^{1/2}u_0\|_{L^2}<\alpha,$$
then $u(t)$ exists globally in time, and
$$\|u_0\|_{L^2}\|\mathcal{H}^{1/2}u(t)\|_{L^2}<\alpha,\quad\forall t\in\mathbb{R}.$$
$(ii)$ If
$$\|u_0\|_{L^2}\|\mathcal{H}^{1/2}u_0\|_{L^2}>\alpha,$$
then
$$\|u_0\|_{L^2}\|\mathcal{H}^{1/2}u(t)\|_{L^2}>\alpha$$
during the maximal existence time.
\end{theorem}

\begin{remark}
Theorem \ref{thm:ULDichotomy} extends the global-versus-finite time dichotomy in the homogeneous case \cite{HR, DHR}, since, if $V=0$, then $\mathcal{ME}=M[Q]E_0[Q]$ and $\alpha=\|Q\|_{L^2}\|\nabla Q\|_{L^2}$.
\end{remark}

\subsection{Criteria for scattering}
The second part of this paper is devoted to investigating the dynamical behavior of global solutions in Theorem \ref{thm:ULDichotomy} $(i)$. In the homogeneous case, Duyckaerts, Holmer and Roudenko \cite{DHR} proved that every global solution in Theorem \ref{thm:ULDichotomy} $(i)$ has finite $S(\dot{H}^{1/2})$ norm (see (2.1)) and, as a consequence, it scatters in $H^1$. Motivated by this work, we formulate the following scattering conjecture for the perturbed equation $(\textup{NLS}_V)$.

\begin{conjecture}[Scattering]\label{Conjecture}
Every global solution satisfying the conditions in Theorem \ref{thm:ULDichotomy} $(i)$ has finite $S(\dot{H}^{1/2})$-norm, and it scatters in $H^1$.
\end{conjecture}

To prove the scattering conjecture, we employ the robust concentration-compactness approach. This method has been developed by Colliander-Keel-Staffilani-Takaoka-Tao for the 3d quintic defocusing nonlinear Schr\"odinger equation and Kenig-Merle for the energy-critical focusing nonlinear Schr\"odinger and wave equations \cite{KM1, KM2}. It has been successfully applied to solve scattering problems in various settings.\\

The method of concentration-compactness can be adapted to $(\textup{NLS}_V)$ as follows. We assume that the scattering conjecture is not true, and the there is a threshold mass-energy $\mathcal{ME}_c$ that is strictly less than $\mathcal{ME}$. Then, we attempt to deduce a contradiction in three steps.\\
\textbf{Step 1.} Construct a special solution $u_c(t)$, called a \textit{minimal blow-up solution}, at the threshold between scattering and non-scattering regimes.\\
\textbf{Step 2.} Prove that the solution $u_c(t)$ is precompact in $H^1$.\\
\textbf{Step 3.} Eliminate a minimal blow-up solution by the localized virial identities and the sharp Gagliardo-Nirenberg inequality.\\

First, assuming that the scattering conjecture is false, we construct a minimal blow-up solution (Step 1) and show that it satisfies the compactness properties (Step 2).
\begin{theorem}[Minimal blow-up solution]\label{thm:MinimalBlowup1} If Conjecture \ref{Conjecture} fails, then there exists a global solution $u_c(t)$ such that
$$M[u_{c,0}]E[u_{c,0}]<\mathcal{ME},\ \|u_{c,0}\|_{L^2}\|\mathcal{H}^{1/2}u_{c,0}\|_{L^2}<\alpha\textup{ and }\|u_c(t)\|_{S(\dot{H}^{1/2})}=\infty,$$
where $u_{c,0}=u_c(0)$. Moreover, $u_c(t)$ is precompact in $H^1$.
\end{theorem}

The proof of Theorem \ref{thm:MinimalBlowup1} depends heavily on linear profile decomposition. However, since a potential perturbation breaks the symmetries of the both linear and the nonlinear Schr\"odinger equation, we need to modify the linear profile decomposition (Proposition \ref{prop:ProfileDecomposition}) and its applications. We remark that similar modifications appear in \cite{KVZ}, where the authors established scattering for the defocusing energy critical nonlinear Schr\"odinger equation in the exterior of a strictly convex obstacle.\\

For the scattering conjecture, we give a partial answer by eliminating a minimal blow-up solution (Step 3), provided that a potential is repulsive.
\begin{theorem}[Scattering, when $V$ is repulsive]\label{thm:Scattering}
Suppose that $V$ satisfies $(1.1)$ and $(1.2)$. We also assume that $V\geq0$, $x\cdot \nabla V(x)\leq0$ and $x\cdot\nabla V\in L^{3/2}$. If
$$M[u_0]E[u_0]<M[Q]E_0[Q], \ \|u_0\|_{L^2}\|\mathcal{H}^{1/2}u_0\|_{L^2}<\|Q\|_{L^2}\|\nabla Q\|_{L^2},$$
then $u(t)$ scatters in $H^1$.
\end{theorem}

To prove Theorem \ref{thm:Scattering}, we terminate a minimal blow-up solution employing the localized virial identity
\begin{equation}
\begin{aligned}
\partial_t^2\int_{\mathbb{R}^3}\chi_R|u|^2dx&=4\sum_{i,j=1}^3\Re\int_{\mathbb{R}^3}\partial_{x_ix_j}\chi_R\partial_{x_i}u\overline{\partial_{x_j}u}dx-\int_{\mathbb{R}^3} \Delta\chi_R|u|^4dx\\
&-\int_{\mathbb{R}^3}\Delta^2\chi_R|u|^2 dx-2\int_{\mathbb{R}^3}(\nabla\chi_R\cdot\nabla V)|u|^2 dx,
\end{aligned}
\end{equation}
where $\chi\in C_c^\infty$ is a radially symmetric function such that $\chi(x)=|x|^2$ for $|x|\leq 1$ and $\chi(x)=0$ for $|x|\geq 2$, and $\chi_R:=R^2\chi(\frac{\cdot}{R})$ for $R>0$ (see Proposition 7.1). To this end, the right hand side of $(1.9)$ has to be coercive. However, it may not be coercive due to the last term in $(1.9)$,
\begin{equation}
-2\int_{\mathbb{R}^3}(\nabla\chi_R\cdot\nabla V)|u|^2 dx= -4\int_{\mathbb{R}^3}(x\cdot\nabla V)|u|^2 dx+o_R(1).
\end{equation}
The repulsive condition guarantees $(1.10)$ to be non-negative.\\

The repulsiveness assumption on the potential $V$ in Theorem \ref{thm:Scattering} is analogous to the convexity of the obstacle $\Omega$ in \cite{KVZ}. In both cases, once wave packets are reflected by a potential or a convex obstacle, they never be refocused. However, unlike the obstacle case, if the confining part of a potential is not strong, then the dynamics of wave packets may not be changed much. Indeed, scattering for the linear equation \eqref{LS} and small data scattering for the nonlinear equation $(\textup{NLS}_V)$ are easy to show under the assumptions $(1.1)$ and $(1.2)$ (Corollary \ref{SmallDataScattering}).\\

An interesting open question is whether the repulsive condition in Theorem \ref{thm:Scattering} is necessary for large data scattering in nonlinear settings. For this question, we address the following remarks.
\begin{remark}
$(i)$ By small modifications of the proofs of our theorems, one can show scattering for a 3d cubic defocusing NLS with a potential
$$i\partial_tu+\Delta u-Vu-|u|^2u=0,\ u(0)=u_0\in H^1,$$
provided that the confining part of the potential $(x\cdot \nabla V(x))_+=\max (x\cdot\nabla V(x), 0)$ is small, precisely
$$\|(x\cdot \nabla V(x))_+\|_{\mathcal{K}}<8\pi$$
(see Theorem \ref{defocusing}).\\
$(ii)$ The repulsive condition is not needed to construct a minimal blow-up solution (Theorem \ref{thm:MinimalBlowup1}). It is used only in the last step to eliminate a minimal blow-up solution by the virial identity.\\
$(iii)$ The integral $(1.10)$ in the localized virial identity is originated from the linear part of the equation $(\textup{NLS}_V)$. Indeed, if $u(t)$ solves the linear Schr\"odingier equation \eqref{LS}, then
\begin{align*}
\partial_t^2\int_{\mathbb{R}^3}\chi_R|u|^2dx&=4\sum_{i,j=1}^3\Re\int_{\mathbb{R}^3}\partial_{x_ix_j}\chi_R\partial_{x_i}u\overline{\partial_{x_j}u}dx-\int_{\mathbb{R}^3}\Delta^2\chi_R|u|^2 dx\\
&-2\int_{\mathbb{R}^3}(\nabla\chi_R\cdot\nabla V)|u|^2 dx.
\end{align*}
Note that scattering for the linear Schr\"odinger equation \eqref{LS} can be obtained without using the virial identities. Thus, the localized virial identity may not be the best tool to eliminate a minimal blow-up.
\end{remark}

\subsection{Organization of the paper} 
In \S 2, we collect preliminary estimates to deal with a linear operator $e^{it(\Delta-V)}$, and record relevant local theories. In \S 3, we solve the variational problem (Proposition \ref{prop:SharpConstantGN}). In \S 4, using the variational problem, we obtain the upper-bound versus lower-bound dichotomy (Theorem \ref{thm:ULDichotomy}). In \S 5-7, we carry out the concentration-compactness argument with several modifications to overcome the broken symmetry. To this end, in \S 5, we establish the linear profile decomposition associated with the scaled linear propagator (Proposition 5.1). Then, we construct a minimal blow-up solution (Theorem \ref{thm:MinimalBlowup1}) in \S 6. Finally, in \S 7, we prove scattering by excluding the minimal blow-up solution, provided that the potential is repulsive (Theorem \ref{thm:Scattering}).

\subsection{Notations} 
We denote by $\textup{NLS}_V(t)u_0$ the solution to $(\textup{NLS}_V)$ with the initial data $u_0$. For $r>0$ and $a\in\mathbb{R}^3$, we define $V_{r,a}:=\tfrac{1}{r^2}V(\tfrac{\cdot-a}{r})$ and $\mathcal{H}_{r,a}:=-\Delta+V_{r,a}$.

\subsection{Acknowledgement}
The author would like to thank his advisor, Justin Holmer, for his help and encouragement. This work was partially supported by the NSF Grant DMS-0901582.
\section{Preliminaries}

\subsection{Strichartz Estimates and norm equivalence}
We record preliminary tools to analyze the perturbed linear propagator $e^{-it\mathcal{H}}=e^{it(\Delta-V)}$.\\

First, we recall the dispersive estimate for the linear propagator $e^{-it\mathcal{H}}$, but for simplicity, we assume that the negative part of a potential is small. 
\begin{lemma}[Dispersive estimate]\label{lem:Dispersion} If $V\in\mathcal{K}_0\cap L^{3/2}$ and $\|V_-\|_{\mathcal{K}}<4\pi$, then
$$\|e^{-it\mathcal{H}}\|_{L^1\to L^\infty}\lesssim |t|^{-3/2}.$$
\end{lemma}

\begin{proof}
By Beceanu-Goldberg \cite{BG}, it suffices to show that $\mathcal{H}$ doesn't have an eigenvalue or a nonnegative resonance. By Lemma A.1, $\mathcal{H}$ is positive, and thus it has no negative eigenvalue. Moreover, by Ionescu-Jerison \cite{IJ}, there is no positive eigenvalue or resonance. 
\end{proof}

By the arguments of Keel-Tao \cite{KT} and Foschi \cite{F} in the abstract setting, one can derive Strichartz estimates from the dispersive estimate and unitarity of the linear propagator $e^{-it\mathcal{H}}$. For notational convenience, we introduce the following definitions. We say that an exponent pair $(q,r)$ is called \textit{$\dot{H}^s$-admissible} (in 3d) if $2\leq q,r\leq\infty$ and 
$$\frac{2}{q}+\frac{3}{r}=\frac{3}{2}-s.$$
We define the Strichartz norm by
$$\|u\|_{S(L^2; I)}:=\sup_{\substack{(q,r):\ L^2\textup{-admissible}\\2\leq q\leq \infty,\,2\leq r\leq 6}} \|u\|_{L^q_{t\in I}L^r_x}$$
and its dual norm by
$$\|u\|_{S'(L^2; I)}:=\inf_{\substack{(q,r):\ L^2\textup{-admissible}\\2\leq \tilde{q}\leq \infty,\,2\leq \tilde{r}\leq 6}} \|u\|_{L^{\tilde{q}'}_{t\in I}L^{\tilde{r}'}_x}.$$
We also define the exotic Strichartz norm by
\begin{equation}
\|u\|_{S(\dot{H}^{1/2}; I)}:=\sup_{\substack{(q,r):\ \dot{H}^{1/2}\textup{-admissible}\\4\leq q\leq \infty,\,3\leq r\leq 6}} \|u\|_{L^q_{t\in I}L^r_x}
\end{equation}
and its dual norm by
$$\|u\|_{S'(\dot{H}^{-1/2}; I)}:=\inf_{\substack{(\tilde{q},\tilde{r}):\ \dot{H}^{-1/2}\textup{-admissible}\\\frac{4}{3}\leq \tilde{q}\leq 2^-,\,3^+\leq \tilde{r}\leq 6}} \|u\|_{L^{\tilde{q}'}_tL^{\tilde{r}'}_x(I\times\mathbb{R}^3)}.$$
Here, $2^-$ is an arbitrarily preselected and fixed number $<2$; similarly for $3^+$. If the time interval $I$ is not specified, we take $I=\mathbb{R}$. 

\begin{remark} The ranges of exponent pairs in the $S(\dot{H}^{1/2})$-norm and the $S'(\dot{H}^{-1/2})$-norm are chosen to satisfy the conditions in Theorem 1.4 of Foschi \cite{F}. Note that $(2,3)$ is not included in $S'(\dot{H}^{-1/2})$, since it is not $H^{-\frac{1}{2}}$-admissible. If $(q,r)=(4,6)$ and $(\tilde{q},\tilde{r})=(\frac{4}{3},6)$, the sharp condition holds. Otherwise, $(q,r)$ and $(\tilde{q},\tilde{r})$ satisfy the non-sharp condition.
\end{remark}

\begin{lemma}[Strichartz estimates]\label{lem:Strichartz} If $V\in\mathcal{K}_0\cap L^{3/2}$ and $\|V_-\|_{\mathcal{K}}<4\pi$, then
\begin{align*}
\|e^{-it\mathcal{H}}f\|_{S(L^2)}&\lesssim\|f\|_{L^2},\\
\Big\|\int_0^t e^{-i(t-s)\mathcal{H}}F(s)ds\Big\|_{S(L^2)}&\lesssim\|F\|_{S'(L^2)}.
\end{align*}
\end{lemma}

\begin{lemma}[Kato inhomogeneous Strichartz estimate]\label{lem:KatoStrichartz} If $V\in\mathcal{K}_0\cap L^{3/2}$ and $\|V_-\|_{\mathcal{K}}<4\pi$, then
$$\Big\|\int_0^t e^{-i(t-s)\mathcal{H}}F(s)ds\Big\|_{S(\dot{H}^{1/2})}\lesssim\|F\|_{S'(\dot{H}^{-1/2})}.$$
\end{lemma}

\begin{remark} Keel-Tao and Foschi assumed the natural scaling symmetry (see (12) of \cite{KT} and Remark 1.5 of  \cite{F}). However, the same proof works without the scaling symmetry. \end{remark}

The following lemma says that the standard Sobolev norms and the Sobolev norms associated with $\mathcal{H}$ are equivalent for some exponent $r$. This norm equivalence lemma is crucial to establish the local theory for the perturbed nonlinear Schr\"odinger equation $(\textup{NLS}_V)$ in Section 2.2.

\begin{lemma}[Norm equivalence]\label{lem:NormEqv} If $V\in\mathcal{K}_0\cap L^{3/2}$ and $\|V_-\|_{\mathcal{K}}<4\pi$, then
\begin{equation}
\|\mathcal{H}^{\frac{s}{2}}f\|_{L^r}\sim\|f\|_{\dot{W}^{s,r}},\ \|(1+\mathcal{H})^{\frac{s}{2}}f\|_{L^r}\sim\|f\|_{W^{s,r}}
\end{equation}
where $1<r<\frac{3}{s}$ and $0\leq s\leq 2$.
\end{lemma}

For the proof, we need the Sobolev inequality associated with $\mathcal{H}$.
\begin{lemma}[Sobolev inequality]\label{lem:Sobolev} If $V\in\mathcal{K}_0\cap L^{3/2}$ and $\|V_-\|_{\mathcal{K}}<4\pi$, then
$$\|f\|_{L^q}\lesssim\|\mathcal{H}^{\frac{s}{2}}f\|_{L^p},\ \|f\|_{L^q}\lesssim\|(1+\mathcal{H})^{\frac{s}{2}}f\|_{L^p}$$
where $1<p<q<\infty$, $1<p<\frac{3}{s}$, $0\leq s\leq 2$ and $\frac{1}{q}=\frac{1}{p}-\frac{s}{3}$.
\end{lemma}

\begin{proof}
Let $a=0$ or $1$. It follows from \cite[Theorem 2]{T} that the heat operator $e^{-t(a+\mathcal{H})}$ obeys the gaussian heat kernel estimate, that is,
$$0\leq e^{-t(a+\mathcal{H})}(x,y)\leq \frac{A_1}{t^{3/2}}e^{-A_2\frac{|x-y|^2}{t}}\quad\forall t>0, \forall x,y\in\mathbb{R}^3$$
for some $A_1,A_2>0$. Applying it to
$$(a+\mathcal{H})^{-\frac{s}{2}}=\frac{1}{\Gamma(s)}\int_0^\infty e^{-t(a+\mathcal{H})}t^{\frac{s}{2}-1}ds,$$
we show that the kernel of $(a+\mathcal{H})^{-\frac{s}{2}}$ satisfies
$$|(a+\mathcal{H})^{-\frac{s}{2}}(x,y)|\lesssim\frac{1}{|x-y|^{3-s}}.$$
This implies that $\|(a+\mathcal{H})^{-\frac{s}{2}}f\|_{L^q}\lesssim\|f\|_{L^p}$ with $p,q,s$ in Lemma \ref{lem:Sobolev}.
\end{proof}

\begin{proof}[Proof of Lemma \ref{lem:NormEqv}]
Let $a=0$ or $1$. We claim that
$$\|(a+\mathcal{H}) f\|_{L^r}\sim \|(a+\Delta) f\|_{L^r},\quad\forall 1<r<\tfrac{3}{2}.$$
Indeed, by H\"older's inequality and the Sobolev inequality, we have
\begin{align*}
\|(a+\mathcal{H})f\|_{L^r}&\leq \|(a-\Delta) f\|_{L^r}+\|Vf\|_{L^r}\\
&\leq  \|(a-\Delta) f\|_{L^r}+\|V\|_{L^{3/2}}\|f\|_{L^{\frac{3r}{3-2r}}}\\
&\lesssim\|(a-\Delta) f\|_{L^r}.
\end{align*}
Similarly, by H\"older's inequality and the Sobolev inequality (Lemma \ref{lem:Sobolev}),
\begin{align*}
\|(a-\Delta) f\|_{L^r}&\leq \|(a+\mathcal{H}) f\|_{L^r}+\|Vf\|_{L^r}\\
&\leq\|(a+\mathcal{H}) f\|_{L^r}+\|V\|_{L^{3/2}}\|f\|_{L^{\frac{3r}{3-2r}}}\\
&\lesssim\|(a+\mathcal{H}) f\|_{L^r}.
\end{align*}
Next, we claim that the imaginary power operator $(a+\mathcal{H})^{iy}$ satisfies
$$\|(a-\Delta)^{iy}\|_{L^r\to L^r}, \|(a+\mathcal{H})^{iy}\|_{L^r\to L^r}\lesssim \la y\ra^{3/2},\quad \forall y\in\mathbb{R}\textup{ and }\forall1<r<\infty.$$
Indeed, since the heat kernel operator $e^{-t\mathcal{H}}$ obeys the gaussian heat kernel estimate (see the proof of Lemma \ref{lem:Sobolev}), these bounds follow from Sikora-Wright \cite{SW}.

Combining the above two claims, we obtain that
\begin{align*}
\|(a+\mathcal{H})^zf\|_{L^r}&\lesssim\la \Im z\ra^{3/2}\|(a-\Delta)^{z}f\|_{L^r},\\
\|(a-\Delta)^{z}f\|_{L^r} &\lesssim\la \Im z\ra^{3/2}\|(a+\mathcal{H})^zf\|_{L^r}
\end{align*}
for $1<r<\infty$ when $\Re z=0$ and for $1<r<\frac{3}{2}$ when $\Re z=1$. Finally, applying the Stein-Weiss complex interpolation, we prove the norm equivalence lemma.
\end{proof}

\begin{remark}The range of exponent $r$ in (2.2) is known to be sharp when $s=1$ \cite{S}.
\end{remark}
 
As an application of Strichartz estimates and the norm equivalence, we obtain the linear scattering.

\begin{lemma}[Linear scattering]\label{linear scattering}
$(i)$ Suppose that $V\in\mathcal{K}_0\cap L^{3/2}$ and $\|V_-\|_{\mathcal{K}}<4\pi$. Then, for any $\psi\in L^2$, there exist $\tilde{\psi}^\pm\in L^2$ such that 
$$\|e^{it\Delta}\psi-e^{-it\mathcal{H}}\tilde{\psi}^\pm\|_{L^2}\to0\textup{ as }t\to\pm\infty.$$
$(ii)$ If we further assume that $V\in W^{1,3/2}$, then for any $\psi\in H^1$, there exist $\tilde{\psi}^\pm\in H^1$ such that 
$$\|e^{it\Delta}\psi-e^{-it\mathcal{H}}\tilde{\psi}^\pm\|_{H^1}\to0\textup{ as }t\to\pm\infty.$$
\end{lemma}
\begin{proof}
$(i)$ Observe that if $u(t)$ solves
$$i\partial_t u+\Delta u=0\Longleftrightarrow i\partial_t u-\mathcal{H}u=-Vu$$
with initial data $\psi$, then it solves the integral equation
$$u(t)=e^{-it\mathcal{H}}\psi-i\int_0^t e^{-i(t-s)\mathcal{H}}(Vu(s)) ds.$$
Applying Strichartz estimates, we obtain 
\begin{align*}
&\|e^{it_1\mathcal{H}}e^{it_1\Delta}\psi-e^{it_2\mathcal{H}}e^{it_2\Delta}\psi\|_{L^2}=\|e^{it_1\mathcal{H}}u(t_1)-e^{it_2\mathcal{H}}u(t_2)\|_{L^2}=\Big\|\int_{t_1}^{t_2} e^{is\mathcal{H}}(Vu(s))ds\Big\|_{L^2}\\
&\lesssim\|Vu(t)\|_{L_{t\in[t_1,t_2]}^2L_x^{6/5}}\lesssim\|V\|_{L^{3/2}}\|u(t)\|_{L_{t\in[t_1,t_2]}^2L_x^{6}}\to0\textup{ as }t_1,t_2\to\pm\infty,
\end{align*}
where in the last step, we used the fact that $\|u(t)\|_{L_{t\in\mathbb{R}}^2L_x^6}=\|e^{it\Delta}\psi\|_{L_{t\in\mathbb{R}}^2L_x^6}\lesssim\|\psi\|_{L^2}<\infty$ (by Strichartz estimates). Hence, the limits
$$\tilde{\psi}^\pm=\lim_{t\to\pm\infty}e^{it\mathcal{H}}e^{it\Delta}\psi$$
exist in $L^2$. Now, repeating the above estimates, we prove that
\begin{align*}
\|e^{it\Delta}\psi-e^{-it\mathcal{H}}\tilde{\psi}^\pm\|_{L^2}&=\|e^{it\mathcal{H}}e^{it\Delta}\psi-\tilde{\psi}^\pm\|_{L^2}=\Big\|\int_t^{\pm\infty} e^{is\mathcal{H}}(Vu(s))ds\Big\|_{L^2}\\
&\lesssim\|Vu(s)\|_{L_{s\in[t,\pm\infty]}^2L_x^{6/5}}\to0\textup{ as }t\to\pm\infty.
\end{align*}
$(ii)$ For scattering in $H^1$, we need to use the norm equivalence lemma, since the linear propagator $e^{it\mathcal{H}}$ and the derivative don't commute. First, by the norm equivalence, we get
\begin{align*}
&\|e^{it_1\mathcal{H}}e^{it_1\Delta}\psi-e^{it_2\mathcal{H}}e^{it_2\Delta}\psi\|_{H^1}\sim\|(1+\mathcal{H})^{1/2}(e^{it_1\mathcal{H}}e^{it_1\Delta}\psi-e^{it_2\mathcal{H}}e^{it_2\Delta}\psi)\|_{L^2}\\
&=\Big\|(1+\mathcal{H})^{1/2}\int_{t_1}^{t_2} e^{is\mathcal{H}}(Ve^{is\Delta}\psi)ds\Big\|_{L^2}=\Big\|\int_{t_1}^{t_2} e^{is\mathcal{H}}(1+\mathcal{H})^{1/2}(Ve^{is\Delta}\psi)ds\Big\|_{L^2}.
\end{align*}
Applying the Strichartz estimates and the norm equivalence again, we obtain that 
\begin{align*}
&\|e^{it_1\mathcal{H}}e^{it_1\Delta}\psi-e^{it_2\mathcal{H}}e^{it_2\Delta}\psi\|_{H^1}\lesssim\|(1+\mathcal{H})^{1/2}(Ve^{it\Delta}\psi)\|_{L_{t\in[t_1,t_2]}^2 L_x^{6/5}}\\
&\sim \|Ve^{it\Delta}\psi\|_{L_{t\in[t_1,t_2]}^2 W_x^{1,6/5}}\lesssim \|V\|_{W^{1,3/2}}\|e^{it\Delta}\psi\|_{L_{t\in[t_1,t_2]}^2 W_x^{1,6}}\to 0
\end{align*}
as $t_1, t_2\to\pm\infty$, since $\|e^{it\Delta}\psi\|_{L_{t\in\mathbb{R}}^2 W_x^{1,6}}\lesssim\|\phi\|_{H^1}$. Therefore, the limits
$$\tilde{\psi}^\pm=\lim_{t\to\pm\infty}e^{it\mathcal{H}}e^{it\Delta}\psi$$
exist in $H^1$. Moreover, repeating the above estimates, we show that $(e^{it\Delta}\psi-e^{-it\mathcal{H}}\tilde{\psi}^\pm)\to 0$ in $H^1$ as $t\to\pm\infty.$
\end{proof}

\begin{remark}[Scaling and spatial translation]\label{rem:Scaling}  Note that the implicit constants for the above estimates are independent of the scaling and translation $V(x)\mapsto V_{r_0,x_0}=\tfrac{1}{r_0^2}V(\tfrac{\cdot-x_0}{r_0})$.
For example, let $c=c(V)>0$ be the sharp constant for Strichartz estimate. Then, by Strichartz estimate for $e^{it(\Delta-V)}$, we have
\begin{align*}
&\|e^{it(\Delta-V_{r_0,x_0})}(f(\tfrac{\cdot-x_0}{r_0}))\|_{L_t^qL_x^r}=\|(e^{i\cdot(-\Delta+V)}f)(\tfrac{t}{r_0^2}, \tfrac{x-x_0}{r_0})\|_{L_t^qL_x^r}=r_0^{\frac{2}{q}+\frac{3}{r}}\|e^{it(-\Delta+V)}f\|_{L_t^qL_x^r}\\
&\leq r_0^{\frac{2}{q}+\frac{3}{r}}c(V)\|f\|_{L_x^2}=r_0^{\frac{2}{q}+\frac{3}{r}-\frac{3}{2}}c(V)\|f(\tfrac{\cdot-x_0}{r_0})\|_{L_x^2}=c(V)\|f(\tfrac{\cdot-x_0}{r_0})\|_{L_x^2}.
\end{align*}
Since $r_0$, $x_0$ and $f$ are arbitrarily chosen, this proves that $c(V_{r_0,x_0})=c(V)$ for all $r_0>0$ and $x_0\in\mathbb{R}^3$.
\end{remark}

\subsection{Local theory}

Now we present the local theory for the perturbed equation $(\textup{NLS}_V)$. We note that the statements and the proofs of the following lemmas are similar to those for the homogeneous equation $(\textup{NLS}_0)$ (see \cite[Section 2]{HR}). The only difference in the proofs is that the norm equivalence (Lemma \ref{lem:NormEqv}) is used in several steps.

\begin{lemma}[Local well-posedness]\label{lem:LWP} $(\textup{NLS}_V)$ is locally well-posed in $H^1$.
\end{lemma}

\begin{proof}
We define $\Phi_{u_0}$ by
$$\Phi_{u_0}(v):=e^{-it\mathcal{H}}u_0+i\int_0^t e^{-i(t-s)\mathcal{H}}(|v|^2v)(s)ds.$$
We claim that 
$$\|\mathcal{H}^{1/2} \Phi_{u_0}(v)\|_{S(L^2; I)}\leq c\|u_0\|_{H^1}+cT^{1/2}\|\mathcal{H}^{1/2}v\|_{S(L^2; I)}^3.
$$
Indeed, by Strichartz estimates and the norm equivalence, we obtain
\begin{align*}
\|\mathcal{H}^{1/2} \Phi_{u_0}(v)\|_{S(L^2; I)}&\lesssim \|\mathcal{H}^{1/2}u_0\|_{L^2}+\|\mathcal{H}^{1/2}(|v|^2 v)\|_{L_{t\in I}^2L_x^{6/5}}\\
&\sim \|u_0\|_{H^1}+\|\la\nabla\ra(|v|^2 v)\|_{L_{t\in I}^2L_x^{6/5}}\quad\textup{(norm equivalence)}\\
&\lesssim \|u_0\|_{H^1}+T^{1/2}\|v\|_{L_{t\in I}^\infty H_x^1}\|v\|_{L_{t\in I}^\infty L_x^6}^2\\
&\leq \|u_0\|_{H^1}+T^{1/2}\|\la\nabla\ra v\|_{L_{t\in I}^\infty L^2}^3\\
&\sim \|u_0\|_{H^1}+T^{1/2}\|\mathcal{H}^{1/2}v\|_{L_{t\in I}^\infty L^2}^3\quad\textup{(norm equivalence)}\\
&\leq \|u_0\|_{H^1}+T^{1/2}\|\mathcal{H}^{1/2}v\|_{S(L^2; I)}^3.
\end{align*}
Similarly, one can show that 
\begin{align*}
&\|\mathcal{H}^{1/2} (\Phi_{u_0}(v_1)-\Phi_{u_0}(v_2))\|_{S(L^2; I)}\\
&\leq cT^{1/2}(\|\mathcal{H}^{1/2}v_1\|_{S(L^2; I)}^2+\|\mathcal{H}^{1/2}v_2\|_{S(L^2; I)}^2)\|\mathcal{H}^{1/2}(v_1-v_2)\|_{S(L^2; I)}.
\end{align*}
Therefore, taking sufficiently small $T>0$, we conclude that $\Phi_{u_0}$ is a contraction on
$$B=\{v: \|\mathcal{H}^{1/2}v\|_{S(L^2)}\leq 2c\|u_0\|_{\dot{H}^{1/2}}\}.$$
\end{proof}

\begin{lemma}[Small data]\label{lem:SmallData} For $A>0$, there exists $\delta_{sd}=\delta_{sd}(A)>0$ such that if $\|u_0\|_{\dot{H}^{1/2}}\leq A$ and $\|e^{-it\mathcal{H}}u_0\|_{S(\dot{H}^{1/2})}\leq\delta_{sd}$, then the solution $u(t)$ is global in $\dot{H}^{1/2}$.
Moreover, 
$$\|u\|_{S(\dot{H}^{1/2})}\leq 2 \|e^{-it\mathcal{H}}u_0\|_{S(\dot{H}^{1/2})},\ \|\mathcal{H}^{1/4}u\|_{S(L^2)}\lesssim\|u_0\|_{\dot{H}^{1/2}}.$$
\end{lemma}

\begin{proof}
Let $\Phi_{u_0}$ be in Lemma \ref{lem:LWP}. By Strichartz estimates and the norm equivalence, 
$$\|\mathcal{H}^{1/4}e^{-it\mathcal{H}}u_0\|_{S(L^2)}\lesssim \|\mathcal{H}^{1/4}u_0\|_{L^2}\sim\|u_0\|_{\dot{H}^{1/2}}.$$
By the Kato Strichartz estimate and the Sobolev inequality (Lemma \ref{lem:Sobolev}), 
\begin{align*}
\Big\|\int_0^t e^{-i(t-s)\mathcal{H}}(|v|^2v)(s)ds\Big\|_{S(\dot{H}^{1/2})}&\lesssim \||v|^2v\|_{L_t^{5/2}L_x^{15/11}}\leq\|v\|_{L_{t,x}^5}^2\|v\|_{L_t^\infty L_x^3}\\
&\lesssim\|v\|_{L_{t,x}^5}^2\|\mathcal{H}^{1/4}v\|_{L_t^\infty L_x^2},
\end{align*}
and by Strichartz estimates, the norm equivalence and the fractional Leibniz rule,
\begin{align*}
\Big\|\mathcal{H}^{1/4}\int_0^t e^{-i(t-s)\mathcal{H}}(|v|^2v)(s)ds\Big\|_{S(L^2)}&\lesssim\|\mathcal{H}^{1/4}(|v|^2v)\|_{L_{t,x}^{10/7}}\sim\||\nabla|^{1/2}(|v|^2v)\|_{L_{t,x}^{10/7}}\\
&\lesssim\|v\|_{L_{t,x}^5}^2\||\nabla|^{1/2}v\|_{L_{t,x}^{10/3}}\sim\|v\|_{L_{t,x}^5}^2\|\mathcal{H}^{1/4}v\|_{L_{t,x}^{10/3}}.
\end{align*}
Therefore, we obtain that
\begin{align*}
\|\Phi_{u_0}(v)\|_{S(\dot{H}^{1/2})}&\leq \|e^{-it\mathcal{H}}u_0\|_{S(\dot{H}^{1/2})}+c\|v\|_{S(\dot{H}^{1/2})}^2\|\mathcal{H}^{1/4}v\|_{S(L^2)},\\
\|\mathcal{H}^{1/4}\Phi_{u_0}(v)\|_{S(L^2)}&\leq c\|u_0\|_{\dot{H}^{1/2}}+c\|v\|_{S(\dot{H}^{1/2})}^2\|\mathcal{H}^{1/4}v\|_{S(L^2)}.
\end{align*}
Now we let $\delta_{sd}=\min(\frac{1}{4\sqrt{c}},\frac{1}{16c^2A})$. Then, $\Phi_{u_0}$ is a contraction on 
$$B=\{v: \|v\|_{S(\dot{H}^{1/2})}\leq2 \|e^{-it\mathcal{H}}u_0\|_{S(\dot{H}^{1/2})},\ \|\mathcal{H}^{1/4}v\|_{S(L^2)}\leq 2c\|u_0\|_{\dot{H}^{1/2}}\}.$$
\end{proof}

It follows from the local well-posedness (Lemma \ref{lem:LWP}) that if a solution is uniformly bounded in $H^1$ during its existence time, then it exists globally in time. However, uniform boundedness is not sufficient for scattering. For instance, in the homogeneous case $(V=0)$, there are infinitely many non-scattering periodic solutions \cite{BL}. The following lemma provides a simple condition for scattering.

\begin{lemma}[Finite $S(\dot{H}^{1/2})$ norm implies scattering]\label{lem:H1Scattering}
Suppose that $u(t)$ is a global solution satisfying 
$$\sup_{t\in\mathbb{R}}\|u(t)\|_{H^1}<\infty.$$
If $u(t)$ has finite $S(\dot{H}^{1/2})$ norm, then $u(t)$ scatters in $H^1$ as $t\to\pm\infty$.
\end{lemma}

\begin{proof}
We define
$$\psi^\pm:=u(0)+i\int_0^{\pm\infty} e^{is\mathcal{H}}(|u|^2u)(s)ds=u(0)+i\lim_{t\to\pm\infty}\int_0^t e^{is\mathcal{H}}(|u|^2u)(s)ds.$$
Indeed, such limits exist in $H^1$, since by the norm equivalence and Strichartz estimates,
\begin{equation}
\begin{aligned}
&\Big\|\int_{t_1}^{t_2}e^{is\mathcal{H}}(|u|^2u)(s)ds\Big\|_{H^1}\sim\Big\|(1+\mathcal{H})^{1/2}\int_{t_1}^{t_2}e^{is\mathcal{H}}(|u|^2u)(s)ds\Big\|_{L_x^2}\\
&\lesssim \|(1+\mathcal{H})^{1/2}(|u|^2u)\|_{L_{[t_1,t_2]}^2L_x^{6/5}}\sim\||u|^2u\|_{L_{[t_1,t_2]}^2W_x^{1,6/5}}\leq\|u\|_{L_t^\infty H_x^1}\|u\|_{L_{[t_1,t_2]}^4L_x^6}^2\to 0
\end{aligned}
\end{equation}
as $t_1, t_2\to\pm\infty$. Hence, $\psi^\pm$ is well-defined. Then, repeating the estimates in $(2.3)$, we conclude that 
\begin{align*}
\|u(t)-e^{-it\mathcal{H}}\psi^{\pm}\|_{H^1}=\Big\|\int_t^{\pm\infty}e^{is\mathcal{H}}(|u|^2u)(s)ds\Big\|_{H^1}\lesssim\cdot\cdot\cdot\lesssim\|u\|_{L_t^\infty H_x^1}\|u\|_{L_{[t,\pm\infty]}^4L_x^6}^2\to0
\end{align*}
as $t\to\pm\infty$.
\end{proof}

\begin{lemma}[Long time perturbation lemma]\label{lem:Perturbation}
For $A>0$, there exist $\epsilon_0=\epsilon_0(A)>0$ and $C=C(A)>0$ such that the following holds: Let $u(t)\in C_t(\mathbb{R}; H_x^1)$ be a solution to $(\textup{NLS}_V)$. Suppose that $\tilde{u}(t)\in C_t(\mathbb{R}; H_x^1)$ is a solution to the perturbed $(\textup{NLS}_V)$
$$i\tilde{u}_t-\mathcal{H}\tilde{u}+|\tilde{u}|^2\tilde{u}=e$$
satisfying
$$\|\tilde{u}\|_{S(\dot{H}^{1/2})}\leq A,\ \|e^{-i(t-t_0)\mathcal{H}}(u(t_0)-\tilde{u}(t_0))\|_{S(\dot{H}^{1/2})}\leq\epsilon_0 \textup{ and } \|e\|_{S'(\dot{H}^{-1/2})}\leq \epsilon_0.$$
Then,
$$\|u\|_{S(\dot{H}^{1/2})}\leq C=C(A).$$
\end{lemma}

\begin{proof}
We omit the proof, since it is similar to that for \cite[Proposition 2.3]{HR}. Indeed, as we observed in the proofs of the previous lemmas, one can easily modify the proof of \cite[Proposition 2.3]{HR} using the norm equivalence (Lemma \ref{lem:NormEqv}).
\end{proof}

\section{Variational Problem}
In this section, we prove Proposition \ref{prop:SharpConstantGN}. Precisely, we will find a maximizer or a maximizing sequence for the nonlinear functional
$$\mathcal{W}_V(u)=\frac{\|u\|_{L^4}^4}{\|u\|_{L^2}\|\mathcal{H}^{1/2} u\|_{L^2}^3}=\frac{\|u\|_{L^4}^4}{\|u\|_{L^2}(\|\nabla u\|_{L^2}^2+\int_{\mathbb{R}^3}V|u|^2 dx)^{3/2}}.$$

\subsection{Nonnegative potential}
We will show Proposition \ref{prop:SharpConstantGN} $(i)$. If $V\geq0$, then one can find a maximizing sequence simply by translating the ground state $Q$ for the nonlinear elliptic equation
$$\Delta Q-Q+Q^3=0.$$
Indeed, the sharp constant for the standard Gagliardo-Nirenberg inequality is given by the ground state $Q$, precisely, 
$$\|u\|_{L^4}^4\leq\frac{\|Q\|_{L^4}^4}{\|Q\|_{L^2}\|\nabla Q\|_{L^2}^3}\|u\|_{L^2}\|\nabla u\|_{L^2}^3\Longleftrightarrow\mathcal{W}_0(Q)\geq\mathcal{W}_0(u).$$
Moreover, we have
$$\lim_{n\to\infty}\mathcal{W}_V(Q(\cdot-n))=\lim_{n\to\infty}\frac{\|Q\|_{L^4}^4}{\|Q\|_{L^2}(\|\nabla Q\|_{L^2}^2+\int_{\mathbb{R}^3}VQ(\cdot-n)^2 dx)^{3/2}}=\mathcal{W}_0(Q).$$
On the other hand, since $V\geq 0$, it is obvious that 
$$\mathcal{W}_0(u)>\mathcal{W}_V(u).$$
Collecting all, we conclude that
$$\lim_{n\to\infty}\mathcal{W}_V(Q(\cdot-n))>\mathcal{W}_V(u),\quad \forall u\in H^1.$$
Therefore, we conclude that $\{Q(\cdot-n)\}_{n=1}^\infty$ is a maximizing sequence for $\mathcal{W}_V(u)$.

\subsection{Potential having a negative part}
We prove Proposition \ref{prop:SharpConstantGN} $(ii)$ by two steps. First, we find a maximizer. Then, we show the properties of the maximizer.

\subsubsection{Maximizer}
We will find a maximizer using the profile decomposition of Hmidi-Keraani \cite{HK}.
\begin{lemma}[Profile decomposition $\textup{\cite[Proposition 3.1]{HK}}$]\label{lem:BubbleDecomp} If $\{u_n\}_{n=1}^\infty$ is a bounded sequence in $H^1$, then there exist a subsequence of $\{u_n\}_{n=1}^\infty$ (still denoted by $\{u_n\}_{n=1}^\infty$), functions $\psi^j\in H^1$ and spatial sequences $\{x_n^j\}_{n=1}^\infty$ such that for $J\geq 1$,
$$u_n=\sum_{j=1}^J\psi^j(\cdot-x_n^j)+R_n^J.$$
The profiles are asymptotically orthogonal: For $j\neq k$,
$$|x_n^j-x_n^{k}|\to\infty\textup{ as } n\to \infty$$
and  for $1\leq j\leq J$,
\begin{equation}
R_n^J(\cdot+x_n^j)\rightharpoonup 0\textup{ weakly in }H^1.
\end{equation}
The remainder sequence is asymptotically small:
$$\lim_{J\to\infty}\limsup_{n\to\infty}\|R_n^J\|_{L^4}=0.$$
Moreover, the decomposition obeys the asymptotic Pythagorean expansion
\begin{align*}
\|u_n\|_{L^2}^2&=\sum_{j=1}^J\|\psi^j\|_{L^2}^2+\|R_n^J\|_{L^2}^2+o_n(1),\\
\|\nabla u_n\|_{L^2}^2&=\sum_{j=1}^J\|\nabla\psi^j\|_{L^2}^2+\|\nabla R_n^J\|_{L^2}^2+o_n(1).
\end{align*}
\end{lemma}

We also use the following elementary lemma.
\begin{lemma}\label{lem:ElementaryIneq}
Let $a_1, a_2, b_1, b_2, c_1, c_2>0$. If there exists $\epsilon\in(0,1)$ such that if $\epsilon<\frac{a_2}{a_1}<\frac{1}{\epsilon}$, then 
$$\frac{c_1+c_2}{(a_1+a_2)^{1/2}(b_1+b_2)^{3/2}}\leq (1-\tfrac{\epsilon}{8})\Big(\max_{i=1,2}\frac{c_i}{a_i^{1/2}b_i^{3/2}}\Big).$$
\end{lemma}

\begin{proof} Let $\alpha=\frac{a_2}{a_1}$ ($\Rightarrow\alpha\in(\epsilon,\tfrac{1}{\epsilon})$) and $\beta=\frac{b_2}{b_1}$. Without loss of generality, we may assume that $\tfrac{c_1}{a_1^{1/2}b_1^{3/2}}\geq \tfrac{c_2}{a_2^{1/2}b_2^{3/2}}$ ($\Rightarrow \frac{c_2}{c_1}\leq \alpha^{1/2}\beta^{3/2}$). Then, we have
$$\frac{c_1+c_2}{(a_1+a_2)^{1/2}(b_1+b_2)^{3/2}}=\frac{c_1}{a_1^{1/2}b_1^{3/2}}\frac{1+c_2/c_1}{(1+\alpha)^{1/2}(1+\beta)^{3/2}}
\leq\frac{c_1}{a_1^{1/2}b_1^{3/2}}\frac{1+\alpha^{1/2}\beta^{3/2}}{(1+\alpha)^{1/2}(1+\beta)^{3/2}}.$$
By the Young's inequality $ab\leq \frac{1}{4}a^4+\frac{3}{4}b^{4/3}$, it follows that
\begin{align*}
&\frac{1+\alpha^{1/2}\beta^{3/2}}{(1+\alpha)^{1/2}(1+\beta)^{3/2}}=\frac{1}{(1+\alpha)^{1/2}(1+\beta)^{3/2}}+\frac{1}{(1+\tfrac{1}{\alpha})^{1/2}(1+\tfrac{1}{\beta})^{3/2}}\\
&\leq\frac{1}{4(1+\alpha)^2}+\frac{3}{4(1+\beta)^2}+\frac{1}{4(1+\tfrac{1}{\alpha})^2}+\frac{3}{4(1+\tfrac{1}{\beta})^2}=\frac{1+\alpha^2}{4(1+\alpha)^2}+\frac{3(1+\beta^2)}{4(1+\beta)^2}\\
&=1-\frac{\alpha}{2(1+\alpha)^2}-\frac{3\beta}{2(1+\beta)^2}\leq1-\frac{\epsilon}{2(1+\tfrac{1}{\epsilon})^2}\leq 1-\frac{\epsilon}{8}.
\end{align*}
\end{proof}

Let $\{u_n\}_{n=1}^\infty$ be a maximizing sequence. Note that Lemma \ref{lem:BubbleDecomp} cannot be directly applied to the sequence $\{u_n\}_{n=1}^\infty$, because $\{u_n\}_{n=1}^\infty$ may not be bounded in $H^1$. Hence, instead of $\{u_n\}_{n=1}^\infty$, we consider the following sequence. For each $n$, we pick $\alpha_n, r_n>0$ such that
\begin{align*}
\|\alpha_n u_n(\tfrac{\cdot}{r_n})\|_{L^2}^2&=\alpha_n^2r_n^3\|u_n\|_{L^2}^2=1,\\
\|\mathcal{H}_{r_n}^{1/2}\alpha_nu_n(\tfrac{\cdot}{r_n})\|_{L^2}^2&=\alpha_n^2r_n\|\mathcal{H}^{1/2}u_n\|_{L^2}=1,
\end{align*}
where $\mathcal{H}_r=-\Delta+\tfrac{1}{r^2}V(\tfrac{\cdot}{r})$. Since $\mathcal{W}_V(\alpha u)=\mathcal{W}_V(u)$, replacing $\{u_n\}_{n=1}^\infty$ by $\{\alpha_n u_n\}_{n=1}^\infty$, we may assume that $\|u_n(\tfrac{\cdot}{r_n})\|_{L^2}=1$ and $\|\mathcal{H}_{r_n}^{1/2}u_n(\tfrac{\cdot}{r_n})\|_{L^2}=1$. Set $\tilde{u}_n=u_n(\tfrac{\cdot}{r_n})$. Then, $\{\tilde{u}_n\}_{n=1}^\infty$ is a bounded sequence in $H^1$, because by the norm equivalence,
$$\|\tilde{u}_n\|_{L^2}^2=1,\ \|\nabla\tilde{u}_n\|_{L^2}^2\sim\|\mathcal{H}_{r_n}^{1/2}\tilde{u}_n\|_{L^2}^2=1.$$
Now, applying Lemma \ref{lem:BubbleDecomp} to $(\tilde{u}_n)$, we write
\begin{equation}
\tilde{u}_n=\sum_{j=1}^J\psi^j(\cdot-x_n^j)+R_n^J.
\end{equation}
\textbf{(Step 1. $\psi^j=0$ for all $j\geq 2$)} We will show that $\psi^j=0$ for all $j\geq 2$. For contradiction, we assume that $\psi^j\neq0$ for some $j\geq 2$.. Extracting a subsequence, we may assume that $r_n\to r_0\in[0,+\infty]$ and $x_n^1\to x_0^1\in\mathbb{R}^3\cup\{\infty\}$. By Lemma \ref{lem:BubbleDecomp}, we have
\begin{equation}
\begin{aligned}
&\frac{1}{2}\|\psi^j\|_{\dot{H}^1}^2\leq\|R_n^1\|_{\dot{H}^1}^2\leq C,\ \frac{1}{2}\|\psi^j\|_{L^2}^2\leq\|R_n^1\|_{L^2}^2\leq C,\ \frac{1}{2}\|\psi^j\|_{L^4}^4\leq\|R_n^1\|_{L^4}^4\leq C
\end{aligned}
\end{equation}
for all sufficiently large $n$. Let
\begin{align*}
&a_1(n)=\|\psi^1\|_{L^2}^2, &&b_1(n)=\|\mathcal{H}_{r_n}^{1/2}(\psi^1(\cdot-x_n^1))\|_{L^2}^2, &&c_1(n)=\|\psi^1\|_{L^4}^4,\\
&a_2(n)=\|R_n^1\|_{L^2}^2, &&b_2(n)=\|\mathcal{H}_{r_n}^{1/2}R_n^1\|_{L^2}^2, &&c_2(n)=\|R_n^1\|_{L^4}^4.
\end{align*}
We claim that
\begin{align}
&\|\tilde{u}_n\|_{L^2}^2=a_1(n)+a_2(n)+o_n(1),\\
&\|\mathcal{H}_{r_n}^{1/2} \tilde{u}_n\|_{L^2}^2=b_1(n)+b_2(n)+o_n(1),\\
&\|\tilde{u}_n\|_{L^4}^4=c_1(n)+c_2(n)+o_n(1).
\end{align}
First, $(3.4)$ follows from the asymptotic Pythagorean expansion in Lemma \ref{lem:BubbleDecomp}. For $(3.5)$, we write
\begin{align*}
\|\mathcal{H}_{r_n}^{1/2}\tilde{u}_n\|_{L^2}^2&=\|\mathcal{H}_{r_n}^{1/2}\psi^1(\cdot-x_n^1)\|_{L^2}^2+\|\mathcal{H}_{r_n}^{1/2}R_n^1\|_{L^2}^2+2\Re\la \nabla\psi^1(\cdot-x_n^1), \nabla R_n^1\ra_{L^2}.\\
&+2\Re\int_{\mathbb{R}^3}V_{r_n}\psi^1(\cdot-x_n^1)\overline{R_n^1}dx.
\end{align*}
By $(3.1)$, the third term is $o_n(1)$. It suffices to show that the last term is $o_n(1)$. If $r_n\to r_0\in(0,+\infty)$ and $x_n^1\to x_0^1\in\mathbb{R}^3$, then
\begin{align*}
&\int_{\mathbb{R}^3}V_{r_n}\psi^1(\cdot-x_n^1)\overline{R_n^1}dx=\int_{\mathbb{R}^3}V_{r_0}(\cdot+x_0^1)\psi^1\overline{R_n^1(\cdot+x_n^1)}dx+o_n(1)\\
&=\la (1-\Delta)^{-1}(V_{r_0}(\cdot+x_0^1)\psi^1),R_n^1(\cdot+x_n^1)\ra_{H^1}+o_n(1)=o_n(1),
\end{align*}
where the last step follows from $(3.2)$ and 
\begin{align*}
\|(1-\Delta)^{-1}(V_{r_0}(\cdot+x_0^1)\psi^1)\|_{H^1}&=\|\la\nabla\ra^{-1}(V_{r_0}(\cdot+x_0^1)\psi^1)\|_{L^2}\lesssim\|V_{r_0}(\cdot+x_0^1)\psi^1\|_{L^{6/5}}\\
&\leq\|V_{r_0}(\cdot+x_0^1)\|_{L^{3/2}}\|\psi^1\|_{L^6}\leq \|V\|_{L^{3/2}}\|\psi^1\|_{H^1}.
\end{align*}
On the other hand, if $r_n\to0$, $r_n\to+\infty$ or $x_n^1\to\infty$, then 
$$\Big|\int_{\mathbb{R}^3}V_{r_n}\psi^1(\cdot-x_n^1)\overline{R_n^1}dx\Big|\leq \|V_{r_n}(\cdot+x_n^1)\psi^1\|_{L^{6/5}} \|R_n^1\|_{L^6}\lesssim\|V_{r_n}(\cdot+x_n^1)\psi^1\|_{L^{6/5}} \|R_n^1\|_{\dot{H}^1}.$$
But, since
$$\|V_{r_n}(\cdot+x_n^1)\psi^1\|_{L^{6/5}}\leq\|V_{r_n}(\cdot+x_n^1)\|_{L^{3/2}}\|\psi^1\|_{L^6}=\|V\|_{L^{3/2}}\|\psi^1\|_{L^6}<\infty,$$
we have
$$\|V_{r_n}(\cdot+x_n^1)\psi^1\|_{L^{6/5}}\to0$$
as $r_n\to 0$, $r_n\to+\infty$ or $x_n\to\infty$. To prove $(3.6)$, given $\epsilon>0$, by the asymptotic smallness of the remainder sequence in Lemma \ref{lem:BubbleDecomp}, one can find $J\gg1$ such that $\|R_n^{J}\|_{L^4}^4\leq\epsilon$ for large $n$. Then, due to the asymptotic orthogonality of profiles, we obtain
\begin{align*}
\|\tilde{u}_n\|_{L^4}^4&=\Big\|\sum_{j=1}^{J}\psi^j(\cdot-x_n^j)+R_n^{J}\Big\|_{L^4}^4\\
&=\|\psi^1\|_{L^4}^4+\Big\|\sum_{j=2}^{J}\psi^j(\cdot-x_n^j)\Big\|_{L^4}^4+\sum_{j=1}^{J}\la |\psi^j|^2\psi^j, R_n^J(\cdot+x_n^j)\ra_{L^2}+o_n(1)+O(\epsilon).
\end{align*}
Observe that 
$$\Big\|\sum_{j=2}^{J}\psi^j(\cdot-x_n^j)\Big\|_{L^4}^4=\|R_n^{J}-R_n^1\|_{L^4}^4+o_n(1)=\|R_n^1\|_{L^4}^4+o_n(1)+O(\epsilon)$$
For each $j$, we choose $\varphi^j\in C_c^\infty$ such that $\|\varphi^j-|\psi^j|^2\psi^j\|_{L^2}\leq \epsilon/J$. This is possible, because $\||\psi^j|^2\psi^j\|_{L^2}=\|\psi^j\|_{L^6}^3\lesssim\|\psi^j\|_{H^1}^3<\infty$. Then, 
\begin{align*}
\sum_{j=1}^J|\la |\psi^j|^2\psi^j, R_n^J(\cdot+x_n^j)\ra_{L^2}|&\leq\sum_{j=1}^J|\la \varphi^j, R_n^J(\cdot+x_n^j)\ra_{L^2}|+O(\epsilon)\to O(\epsilon).
\end{align*}
Therefore, we get
$$\|\tilde{u}_n\|_{L^4}^4=\|\psi^1\|_{L^4}^4+\|R_n^1\|_{L^4}^4+o_n(1)+O(\epsilon).$$
Since $\epsilon>0$ is arbitrary, this proves $(3.6)$.

By Lemma 3.2, it follows from $(3.3)$, $(3.4)$, $(3.5)$ and $(3.6)$ that
$$\lim_{n\to\infty}\mathcal{W}_V(u_n)=\lim_{n\to\infty}\frac{\|\tilde{u}_n\|_{L^4}^4}{\|\tilde{u}_n\|_{L^2}\|\mathcal{H}_{r_n}^{1/2}\tilde{u}_n\|_{L^2}^3}=\lim_{n\to\infty}\frac{c_1(n)+c_2(n)}{(a_1(n)+a_2(n))^{1/2}(b_1(n)+b_2(n))^{3/2}}$$
is strictly less than 
$$\lim_{n\to\infty}\mathcal{W}_V(\psi^1(r_n\cdot-x_n^1))=\lim_{n\to\infty}\frac{\|\psi^1(\cdot-x_n^1)\|_{L^4}^4}{\|\psi^1(\cdot-x_n^1)\|_{L^2}\|\mathcal{H}_{r_n}^{1/2}\psi^1(\cdot-x_n^1)\|_{L^2}^3}=\lim_{n\to\infty}\frac{c_1(n)}{a_1(n)^{1/2}b_1(n)^{3/2}}$$
or
$$\lim_{n\to\infty}\mathcal{W}_V(R_n^1(r_n\cdot))=\lim_{n\to\infty}\frac{\|R_n^1\|_{L^4}^4}{\|R_n^1\|_{L^2}\|\mathcal{H}_{r_n}^{1/2}R_n^1\|_{L^2}^3}=\lim_{n\to\infty}\frac{c_2(n)}{a_2(n)^{1/2}b_2(n)^{3/2}}.$$
This contradicts to the maximality of $\{u_n\}_{n=1}^\infty$.\\
\textbf{(Step 2. $R_n^1\to 0$ in $H^1$)} Passing to a subsequence, we may assume that $\underset{n\to\infty}\lim\|R_n^1\|_{H^1}$ exists. For contradiction, we assume that 
\begin{equation}
\underset{n\to\infty}\lim\|R_n^1\|_{H^1}>0.
\end{equation}
As in the proof of $(3.5)$, one can show that
\begin{equation}
\int_{\mathbb{R}^3}\mathcal{H}_{r_n}(\psi^1(\cdot-x_n^1))\overline{R_n^1} dx\to 0.
\end{equation}
Moreover, by the asymptotic smallness of the remainder in Lemma \ref{lem:BubbleDecomp}, passing to a subsequence, we have $\|R_n^1\|_{L^4}\to 0$. Therefore, we get
\begin{align*}
c_{GN}&=\lim_{n\to\infty}\mathcal{W}_V(u_n)=\lim_{n\to\infty}\frac{\|u_n\|_{L^4}^4}{\|u_n\|_{L^2}\|\mathcal{H}^{1/2}u_n\|_{L^2}^3}=\lim_{n\to\infty}\frac{\|\tilde{u}_n\|_{L^4}^4}{\|\tilde{u}_n\|_{L^2}\|\mathcal{H}_{r_n}^{1/2}\tilde{u}_n\|_{L^2}^3}\\
&<\lim_{n\to\infty}\frac{\|\psi^1(\cdot-x_n^1)\|_{L^4}^4}{\|\psi^1(\cdot-x_n^1)\|_{L^2}\|\mathcal{H}_{r_n}^{1/2}(\psi^1(\cdot-x_n^1))\|_{L^2}^3}\textup{ (by Step 1, $(3.7)$ and $(3.8)$)}\\
&=\lim_{n\to\infty}\frac{\|\psi^1(r_n\cdot-x_n^1)\|_{L^4}^4}{\|\psi^1(r_n\cdot-x_n^1)\|_{L^2}\|\mathcal{H}^{1/2}\psi^1(r_n\cdot-x_n^1)\|_{L^2}^3}=\lim_{n\to\infty}\mathcal{W}_V(\psi^1(\tfrac{\cdot}{r_n}-x_n^1)),
\end{align*}
which contradicts maximality of $\{u_n\}_{n=1}^\infty$. Therefore, we should have $R_n^1\to 0$ in $H^1$.\\
\textbf{(Step 3. Convergence of $\{x_n\}_{n=1}^\infty$ and $\{r_n\}_{n=1}^\infty$)} So far, we proved that, passing to a subsequence,
$$u_n(x)=\psi(r_nx-x_n),$$
where $r_n\to r_0\in[0,+\infty]$ and $x_n\to x_0\in\mathbb{R}^3\cup\{\infty\}$. Suppose that $r_n\to0$, $r_n\to+\infty$ or $x_n\to\infty$. Then, by the ``free" Gagliardo-Nirenberg inequality and the assumption, we have
\begin{align*}
\frac{\|Q\|_{L^4}^4}{\|Q\|_{L^2}\|\nabla Q\|_{L^2}^3}&\geq\frac{\|\psi\|_{L^4}^4}{\|\psi\|_{L^2}\|\nabla \psi\|_{L^2}^3}=\lim_{n\to\infty}\mathcal{W}_V(\psi(r_n\cdot-x_n))=\lim_{n\to\infty}\mathcal{W}_V(u_n).
\end{align*}
On the other hand, since $V_-\neq0$, there exist $x_*\in\mathbb{R}^3$ and a small $\epsilon>0$ such that $\int_{\mathbb{R}^3} VQ^2(\tfrac{\cdot-x_*}{\epsilon})dx<0$. Thus, it follows that 
$$\mathcal{W}_V(Q(\tfrac{x-x_*}{\epsilon}))>\frac{\|Q(\tfrac{x-x_*}{\epsilon})\|_{L^4}^4}{\|Q(\tfrac{x-x_*}{\epsilon})\|_{L^2}\|\nabla Q(\tfrac{x-x_*}{\epsilon})\|_{L^2}^3}=\frac{\|Q\|_{L^4}^4}{\|Q\|_{L^2}\|\nabla Q\|_{L^2}^3}.$$
Combining two inequalities, we deduce a contradiction.\\
\textbf{(Step 4. Find $\mathcal{Q}$)} Replacing $\psi(r_0\cdot-x_0)$ by $\psi$, we say that $\psi$ is a maximizer of $\mathcal{W}_V(u)$. Then, it solves the Euler-Lagrange equation
equivalently,
$$\la\mathcal{H}\psi-\frac{\|\mathcal{H}^{1/2}\psi\|_{L^2}^2}{3\|\psi\|_{L^2}^2} \psi-\frac{4\|\mathcal{H}^{1/2}\psi\|_{L^2}^2}{3\|\psi\|_{L^4}^4}|\psi|^2\psi,v\ra=0$$
for all $v\in H^1$. We set
$$\mathcal{Q}:=\frac{2\|\mathcal{H}^{1/2}\psi\|_{L^2}}{\sqrt{3}\|\psi\|_{L^4}^2}\psi.$$
Then, $\mathcal{Q}$ is a weak solution to the ground state equation $(1.6)$. We claim that $\mathcal{Q}$ is a strong solution. Indeed, by $(1.6)$ and the H\"older inequality, we have
$$|\la\mathcal{H}\mathcal{Q}, v\ra_{L^2}|=\omega_\mathcal{Q}^2|\la\mathcal{Q}, v\ra_{L^2}|+|\la |\mathcal{Q}|^2\mathcal{Q}, v\ra_{L^2}|\leq\omega_\mathcal{Q}^2\|\mathcal{Q}\|_{L^2}\|v\|_{L^2}+\|\mathcal{Q}\|_{L^6}^3\|v\|_{L^2}\lesssim\|v\|_{L^2}.$$
Hence, we conclude that $(1.6)$ holds in $L^2$.

\subsubsection{Pohozhaev identities}
For $\omega>0$, let $Q_\omega$ be a strong solution to
\begin{equation}
(-\Delta+V)Q_\omega+\omega^2 Q_\omega-|Q_\omega|^2Q_\omega=0.
\end{equation}
Multiplying $(3.9)$ by $Q_\omega$ (and $(x\cdot\nabla Q_\omega)$), integrating and applying integration by parts, we get
\begin{align*}
\left\{\begin{aligned}
&\|\mathcal{H}^{1/2}Q_\omega\|_{L^2}^2+\omega^2\|Q_\omega\|_{L^2}^2-\|Q_\omega\|_{L^4}^4=0,\\
&\|\mathcal{H}^{1/2}Q_\omega\|_{L^2}^2+3\omega^2\|Q_\omega\|_{L^2}^2-\frac{3}{2}\|Q_\omega\|_{L^4}^4+\int_{\mathbb{R}^3}(2V+(x\cdot\nabla V))|Q_\omega|^2 dx=0.
\end{aligned}
\right.
\end{align*}
Solving it as a system of equations for $\|\mathcal{H}^{1/2}Q_\omega\|_{L^2}^2$ and $\|Q_\omega\|_{L^4}^4$, we obtain
\begin{equation}
\|\mathcal{H}^{1/2}Q_\omega\|_{L^2}^2=3\omega^2\|Q_\omega\|_{L^2}^2+\textup{(extra term)},\ \|Q_\omega\|_{L^4}^4=4\omega^2\|Q_\omega\|_{L^2}^2+\textup{(extra term)}.
\end{equation}
where $\textup{(extra term)}=\int_{\mathbb{R}^3}(4V+2(x\cdot\nabla V))|Q_\omega|^2 dx$.

\begin{remark}
If $V=0$, then $\|\nabla Q_\omega\|_{L^2}^2=3\omega^2\|Q_\omega\|_{L^2}^2$ and $\|Q_\omega\|_{L^4}^4=4\omega^2\|Q_\omega\|_{L^2}^2$.
\end{remark}

\begin{proposition}[Pohozhaev identities] Let $\mathcal{Q}$ be the ground state given in Proposition \ref{prop:SharpConstantGN}. Then, 
\begin{equation}\label{eq:Pohozhaev}
\|\mathcal{H}^{1/2}\mathcal{Q}\|_{L^2}^2=3\omega_\mathcal{Q}^2\|\mathcal{Q}\|_{L^2}^2,\ \|\mathcal{Q}\|_{L^4}^4=4\omega_\mathcal{Q}^2\|\mathcal{Q}\|_{L^2}^2.
\end{equation}
\end{proposition}

\begin{proof}
Plugging $\omega_\mathcal{Q}=\frac{\|\mathcal{H}^{1/2}\mathcal{Q}\|_{L^2}}{\sqrt{3}\|\mathcal{Q}\|_{L^2}}$ into $(3.10)$, we see that the extra term should be zero.
\end{proof}
\section{Criteria for Global Well-posedness}

We find the criteria for global well-posedness (Theorem \ref{thm:ULDichotomy}), and obtain properties of such global solutions.

\subsection{Criteria for global well-posedness}
We prove Theorem \ref{thm:ULDichotomy}. By Proposition \ref{prop:SharpConstantGN} and the Pohozaev identities, we prove that if $V$ is nonnegative,
\begin{align*}
\mathcal{ME}&=\|Q\|_{L^2}^2\Big(\frac{1}{2}\|\nabla Q\|_{L^2}^2-\frac{1}{4}\|Q\|_{L^4}^4\Big)=\frac{1}{2}\|Q\|_{L^2}^4=\frac{1}{6}\|Q\|_{L^2}^2\|\nabla Q\|_{L^2}^2=\frac{\alpha^2}{6},\\
c_{GN}&=\frac{\|Q\|_{L^4}^4}{\|Q\|_{L^2}\|\nabla Q\|_{L^2}^3}=\frac{4}{3\|Q\|_{L^2}\|\nabla Q\|_{L^2}}=\frac{4}{3\alpha},
\end{align*}
but if $V$ has nontrivial negative part,
\begin{align*}
\mathcal{ME}&=\|\mathcal{Q}\|_{L^2}^2\Big(\frac{1}{2}\|\mathcal{H}^{1/2} \mathcal{Q}\|_{L^2}^2-\frac{1}{4}\|\mathcal{Q}\|_{L^4}^4\Big)=\frac{1}{2}\|\mathcal{Q}\|_{L^2}^4=\frac{1}{6}\|\mathcal{Q}\|_{L^2}^2\|\mathcal{H}^{1/2}\mathcal{Q}\|_{L^2}^2=\frac{\alpha^2}{6},\\
c_{GN}&=\frac{\|\mathcal{Q}\|_{L^4}^4}{\|\mathcal{Q}\|_{L^2}\|\mathcal{H}^{1/2}\mathcal{Q}\|_{L^2}^3}=\frac{4}{3\|\mathcal{Q}\|_{L^2}\|\mathcal{H}^{1/2} \mathcal{Q}\|_{L^2}}=\frac{4}{3\alpha},
\end{align*}
Then, it follows from the Gagliardo-Nirenberg inequality and the energy conservation law that
\begin{align*}
\mathcal{ME}&>M[u_0]E[u_0]=M[u_0]E[u(t)]=\|u_0\|_{L^2}^2\Big(\frac{1}{2}\|\mathcal{H}^{1/2}u(t)\|_{L^2}^2-\frac{1}{4}\|u(t)\|_{L^4}^4\Big)\\
&\geq\|u_0\|_{L^2}^2\Big(\frac{1}{2}\|\mathcal{H}^{1/2}u(t)\|_{L^2}^2-\frac{1}{4}c_{GN} \|u_0\|_{L^2}\|\mathcal{H}^{1/2}u(t)\|_{L^2}^3\Big)=f((g(t)),
\end{align*}
where $f(x)=\frac{x^2}{2}-\frac{x^3}{3\alpha}$ and $g(t)=\|u_0\|_{L^2}\|\mathcal{H}^{1/2}u(t)\|_{L^2}$. Observe that $f(x)$ is concave for $x\geq0$ and it has a unique maximum at $x=\alpha$, $f(\alpha)=\frac{\alpha^2}{6}=\mathcal{ME}$. Moreover, by $H^1$-continuity of solutions to $(\textup{NLS}_V)$, $g(t)$ is continuous. Therefore, we conclude that either $g(t)<\alpha$ or $g(t)>\alpha$.

\subsection{Properties of global solutions}
We prove important properties of solutions obeying assumptions in Theorem \ref{thm:ULDichotomy} $(i)$.
\begin{lemma}[Comparability of gradient and energy]\label{Comparability} In the situation of Theorem \ref{thm:ULDichotomy} $(i)$, we have
$$2E[u_0]\leq\|\mathcal{H}^{1/2}u(t)\|_{L^2}^2\leq 6E[u_0],\quad\forall t\in\mathbb{R}.$$
\end{lemma}

\begin{proof}
The first inequality is trivial. For the second inequality, by the energy conservation law, we obtain
$$E[u_0]=E[u(t)]=\frac{1}{2}\|\mathcal{H}^{1/2}u(t)\|_{L^2}^2-\frac{1}{4}\|u(t)\|_{L^4}^4\leq\frac{1}{2}\|\mathcal{H}^{1/2}u(t)\|_{L^2}^2.$$
By the Gagliardo-Nirenberg inequality (with $c_{GN}=\frac{4}{3\alpha}$) and Theorem \ref{thm:ULDichotomy} $(i)$, we obtain
$$\|u(t)\|_{L^4}^4\leq\frac{4}{3\alpha}\|u(t)\|_{L^2}\|\mathcal{H}^{1/2}u(t)\|_{L^2}^3\leq\frac{4}{3}\|\mathcal{H}^{1/2}u(t)\|_{L^2}^2.$$
Therefore, by the energy conservation law, we conclude that 
$$E[u_0]=E[u(t)]=\frac{1}{2}\|\mathcal{H}^{1/2}u(t)\|_{L^2}^2-\frac{1}{4}\|u(t)\|_{L^4}^4\geq\frac{1}{6}\|\mathcal{H}^{1/2}u(t)\|_{L^2}^2.$$
\end{proof}

\begin{corollary}[Small data scattering\label{SmallDataScattering}] If $\|u_0\|_{H^1}$ is sufficiently small, then $u(t)=\textup{NLS}_V(t)u_0$ scatters in $H^1$ as $t\to\pm\infty$.
\end{corollary}

\begin{proof}
By Lemma 4.1 and the norm equivalence, we have
\begin{align*}
M[u_0]E[u_0]\sim \|u_0\|_{L^2}^2\|\mathcal{H}^{1/2}u_0\|_{L^2}^2\lesssim\|u_0\|_{H^1}^4\ll1.
\end{align*}
Hence, it follows from Theorem \ref{thm:ULDichotomy} that $u(t)$ is global, and $\|u(t)\|_{H^1}$ is uniformly bounded. Moreover, by Strichartz estimates and the norm equivalence,
$$\|e^{-it\mathcal{H}}u_0\|_{S(\dot{H}^{1/2})}\lesssim \|u_0\|_{\dot{H}^{1/2}}\ll1.$$
By Lemma \ref{lem:SmallData}, this implies that $\|u(t)\|_{S(\dot{H}^{1/2})}<\infty$. Thus, by Lemma 2.12, we conclude that $u(t)$ scatters in $H^1$.
\end{proof}

\begin{proposition}[Existence of wave operators]\label{prop:WaveOperator} If
$$\frac{1}{2}\|\psi^\pm\|_{L^2}\|\mathcal{H}^{1/2}\psi^\pm\|_{L^2}<\mathcal{ME},$$
then there exists unique $u_0\in H^1$, obeying the assumptions in Theorem \ref{thm:ULDichotomy} $(i)$, such that
\begin{equation}
\lim_{t\to\pm\infty}\|\textup{NLS}_V(t)u_0-e^{-it\mathcal{H}}\psi^\pm\|_{H^1}=0.
\end{equation}
\end{proposition}

\begin{proof}
For sufficiently small $\epsilon>0$, choose $T\gg1$ such that $\|e^{-it\mathcal{H}}\psi^+\|_{S(\dot{H}^{1/2};[T,+\infty))}\leq\epsilon$. Then, as we proved in Lemma \ref{lem:SmallData}, one can show that the integral equation
$$u(t)=e^{-it\mathcal{H}}\psi^+-i\int_t^{+\infty} e^{-it\mathcal{H}}(|u|^2 u)(s)ds$$
has a unique solution such that $\|\la\nabla\ra u\|_{S(L^2; [T,+\infty))}\leq 2\|\psi^+\|_{H^1}$ and $\|u\|_{S(\dot{H}^{1/2}; [T,+\infty))}\leq2\epsilon$. Observe that by Strichartz estimates and the norm equivalence,
\begin{align*}
&\|u(t)-e^{-it\mathcal{H}}\psi^+\|_{L_{t\in[T,+\infty)}^\infty H_x^1}\leq\Big\|\int_t^{+\infty} e^{-it\mathcal{H}}(|u|^2 u)(s)ds\Big\|_{L_{t\in[T,+\infty)}^\infty H_x^1}\\
&\lesssim \||u|^2 u\|_{L_{t\in[T,+\infty)}^{10/7}W_x^{1,10/7}}\lesssim \|u\|_{L_{t\in[T,+\infty)}^{10/3}W_x^{1,10/3}}\|u\|_{L_{t\in[T,+\infty)}^5L_x^5}^2\leq2\|\psi^+\|_{H^1}(2\epsilon)^2.
\end{align*}
Since $\epsilon>0$ is arbitrarily small, this proves that $\|u(t)-e^{-it\mathcal{H}}\psi^+\|_{H^1}\to0$ as $t\to+\infty$. By the energy conservation law and Lemma 4.1, we obtain that 
\begin{align*}
&M[u(T)]E[u(T)]=\lim_{t\to+\infty} M[u(t)]E[u(t)]=\lim_{t\to+\infty} M[e^{-it\mathcal{H}}\psi^+]E[e^{-it\mathcal{H}}\psi^+]\\
&=\lim_{t\to+\infty}\|\psi^+\|_{L^2}^2\Big(\frac{1}{2}\|\mathcal{H}^{1/2}\psi^+\|_{L^2}^2-\frac{1}{4}\|e^{-it\mathcal{H}}\psi^+\|_{L^4}^4\Big)\leq\frac{1}{2}\|\psi^+\|_{L^2}^2\|\mathcal{H}^{1/2}\psi^+\|_{L^2}^2<\mathcal{ME}.
\end{align*}
Moreover, we have
\begin{align*}
&\lim_{t\to+\infty}\|u(t)\|_{L^2}^2\|\mathcal{H}^{1/2}u(t)\|_{L^2}^2=\|e^{-it\mathcal{H}}\psi^+\|_{L^2}^2\|\mathcal{H}^{1/2}e^{-it\mathcal{H}}\psi^+\|_{L^2}^2\\
&=\|\psi^+\|_{L^2}^2\|\mathcal{H}^{1/2}\psi^+\|_{L^2}^2<2\mathcal{ME}=\frac{\alpha^2}{3}<\alpha^2.
\end{align*}
Hence, for sufficiently large $T$, $u(T)$ satisfies the assumptions in Theorem \ref{thm:ULDichotomy} $(i)$, which implies that $u(t)$ is a global solution in $H^1$. Let $u_0=u(0)$. Then, $u(t)=\textup{NLS}_V(t)u_0$ satisfies $(4.1)$ for positive time.  By the same way, one can show $(4.1)$ for negative time.
\end{proof}
\section{Linear Profile Decomposition associate with a Perturbed Linear Propagator}

We establish the linear profile decomposition associated with a perturbed linear propagator. This profile decomposition will play a crucial role in construction of a minimal blow-up solution.

\begin{proposition}[Linear profile decomposition]\label{prop:ProfileDecomposition}
Suppose that $r_n=1$, $r_n\to 0$ or $r_n\to\infty$. If $\{u_n\}_{n=1}^\infty$ is a bounded sequence in $H^1$, then there exist a subsequence of $\{u_n\}_{n=1}^\infty$ (still denoted by $\{u_n\}_{n=1}^\infty$), functions $\psi^j\in H^1$, time sequences $\{t_n^j\}_{n=1}^\infty$ and spatial sequences $\{x_n^j\}_{n=1}^\infty$ such that for every $J\geq 1$, 
\begin{equation}\label{eq:ProfileDecomposition}
u_n=\sum_{j=1}^J e^{it_n^j\mathcal{H}_{r_n}}(\psi^j(\cdot-x_n^j))+R_n^J.
\end{equation}
The time sequences and the spatial sequences have the following properties. For every $j$, 
\begin{equation}
t_n^j=0\textup{ or }t_n^j\to\infty, \textup{ and }x_n^j=0\textup{ or }x_n^j\to\infty.
\end{equation}
For every $j\neq k$,
\begin{equation}
t_n^j-t_n^k=0\textup{ or }t_n^j-t_n^k\to\infty\textup{, and }x_n^j-x_n^k=0\textup{ or }x_n^j-x_n^k\to\infty.
\end{equation}
The profiles in $(5.1)$ are asymptotically orthogonal each other: For every $j\neq k$, 
\begin{equation}
|t_n^j-t_n^k|+|x_n^j-x_n^k|\to\infty
\end{equation}
and for $1\leq j\leq J$, 
\begin{equation}
(e^{-it_n^j\mathcal{H}_{r_n}}R_n^J)(\cdot+x_n^j)\rightharpoonup 0\textup{ weakly in }H^1.
\end{equation}
The remainder sequence is asymptotically small: 
\begin{equation}\label{eq:SmallError}
\lim_{J\to\infty}\Big[\lim_{n\to\infty}\|e^{-it\mathcal{H}_{r_n}}R_n^J\|_{S(\dot{H}^{1/2})}\Big]=0.
\end{equation}
Moreover, we have the asymptotic Pythagorean expansion:
\begin{align}
\|u_n\|_{L^2}^2&=\sum_{j=1}^J\|\psi^j\|_{L^2}^2+\|R_n^J\|_{L^2}^2+o_n(1),\\
\|\mathcal{H}_{r_n}^{1/2}u_n\|_{L^2}^2&=\sum_{j=1}^J\|\mathcal{H}_{r_n}^{1/2}(\psi^j(\cdot-x_n^j))\|_{L^2}^2+\|\mathcal{H}_{r_n}^{1/2}R_n^J\|_{L^2}^2+o_n(1).
\end{align}
\end{proposition}

First, we prove the profile decomposition in the case that the potential $V$ effectively disappears by scaling .

\begin{proof}[Proof of Proposition \ref{prop:ProfileDecomposition} when $r_n\to 0$ or $r_n\to+\infty$]
By the profile decomposition associated with the free linear propagator \cite[Proposition]{DHR}, $\{u_n\}_{n=1}^\infty$ has a subsequence (but still denoted by $\{u_n\}_{n=1}^\infty$) such that
\begin{equation}
u_n=\sum_{j=1}^J e^{-it_n^j\Delta}(\psi^j(\cdot-x_n^j))+R_n^J=\sum_{j=1}^J (e^{-it_n^j\Delta}\psi^j)(\cdot-x_n^j)+R_n^J
\end{equation}
satisfying the properties in Proposition \ref{prop:ProfileDecomposition} with $V=0$. Note that in $(5.9)$, we may assume that time sequences $\{t_n^j\}_{n=1}^\infty$ and spatial sequences $\{x_n^j\}_{n=1}^\infty$ satisfy $(5.2)$ and $(5.3)$. Indeed, passing to a subsequence, we may assume that $t_n^j\to t_*^j\in\mathbb{R}\cup\{\infty\}$ and $x_n^j\to x_*^j\in\mathbb{R}^3\cup\{\infty\}$. If $t_*^j\neq\infty$ ($x_*^j\neq\infty$, resp), we replace $e^{-it_n^j\Delta}\psi^j$ ($\psi^j(\cdot-x_n^j)$, resp) in $(5.9)$ by $e^{-it_*^j\Delta}\psi^j$ ($\psi^j(\cdot-x_*^j)$, resp). Then, this modified profile decomposition satisfies $(5.2)$ as well as other properties in Proposition \ref{prop:ProfileDecomposition}. Similarly, one can also modify $(5.9)$ so that $(5.3)$ holds.

Now, replacing $e^{-it\Delta}$ by $e^{it\mathcal{H}_{r_n}}$, we write the profile decomposition 
\begin{equation}
u_n=\sum_{j=1}^J e^{it_n^j\mathcal{H}_{r_n}}(\psi^j(\cdot-x_n^j))+\tilde{R}_n^J,
\end{equation}
where
$$\tilde{R}_n^J=R_n^J+\sum_{j=1}^J e^{it_n^j\mathcal{H}_{r_n}}(\psi^j(\cdot-x_n^j))-e^{-it_n^j\Delta}(\psi^j(\cdot-x_n^j)).$$
We claim that $(5.10)$ has the desired properties. We will show $(5.6)$ only. Indeed, the other properties can be checked easily by the properties obtained from $(5.9)$. To this end, we observe that $u(t)=e^{it\Delta}u_0$ solves the integral equation
\begin{equation}
e^{it\Delta}u_0=e^{it(\Delta-V)}u_0+i\int_0^t e^{i(t-s)(\Delta-V)}(Ve^{is\Delta}u_0) ds.
\end{equation}
Applying $(5.11)$ to $e^{-it\mathcal{H}_{r_n}}R_n^J=e^{it(\Delta-V_{r_n})}R_n^J$, we get
\begin{align*}
\|e^{-it\mathcal{H}_{r_n}}R_n^J\|_{S(\dot{H}^{1/2})}&\leq\|e^{it\Delta}R_n^J\|_{S(\dot{H}^{1/2})}+\Big\|\int_0^t e^{i(t-s)(\Delta-V_{r_n})}(V_{r_n} e^{is\Delta}R_n^J)ds\Big\|_{S(\dot{H}^{1/2})}\\
&\lesssim\|e^{it\Delta}R_n^J\|_{S(\dot{H}^{1/2})}+\|V_{r_n}e^{is\Delta }R_n^J \|_{L_t^4 L_x^{6/5}}\textup{ (by Lemma \ref{lem:KatoStrichartz})}\\
&\leq\|e^{it\Delta}R_n^J\|_{S(\dot{H}^{1/2})}+\|V_{r_n}\|_{L^{3/2}}e^{is\Delta }R_n^J \|_{L_t^4 L_x^6}\\
&=(1+\|V\|_{L^{3/2}})\|e^{it\Delta }R_n^J \|_{S(\dot{H}^{1/2})}\to0
\end{align*}
as $n\to\infty$ and $J\to\infty$. Similarly, we have
\begin{equation}
\begin{aligned}
&\|e^{-it\mathcal{H}_{r_n}}(e^{it_n^j\mathcal{H}_{r_n}}(\psi^j(\cdot-x_n^j))-e^{-it_n^j\Delta}(\psi^j(\cdot-x_n^j)) )\|_{S(\dot{H}^{1/2})}\\
&=\Big\|\int_{-t_n^j}^0 e^{-i(t+t_n^j+s)\mathcal{H}_{r_n}}\Big(V_{r_n}e^{is\Delta}(\psi^j(\cdot-x_n^j))\Big)ds\Big\|_{S(\dot{H}^{1/2})}\\
&\lesssim\|V_{r_n}e^{-it\Delta}(\psi^j(\cdot-x_n^j))\|_{L_t^4L_x^{6/5}}=\|V_{r_n}(\cdot+x_n^j)e^{-it\Delta}\psi^j\|_{L_t^4L_x^{6/5}}\to 0,
\end{aligned}
\end{equation}
where the last step follows from
$$\|V_{r_n}(\cdot+x_n^j)e^{-it\Delta}\psi^j\|_{L_t^4L_x^{6/5}}\leq \|V_{r_n}(\cdot+x_n^j)\|_{L^{3/2}}\|e^{-it\Delta}\psi^j\|_{L_t^4L_x^6}=\|V\|_{L^{3/2}}\|\psi\|_{H^{1/2}}<\infty$$
and the assumption $r_n\to0$ or $r_n\to+\infty$. Thus, we conclude that $\tilde{R}_n^J$ has the asymptotic smallness property $(5.6)$. 
\end{proof}

We give two proofs in the case that $r_n=1$. The first one is simpler but it requires more regularity.

\begin{proof}[Proof of Proposition \ref{prop:ProfileDecomposition} when $r_n=1$, assuming that $V\in W^{1,3/2}$]
As above, we start from the profile decomposition $(5.9)$:
$$u_n=\sum_{j=1}^J e^{-it_n^j\Delta}(\psi^j(\cdot-x_n^j))+R_n^J=\sum_{j=1}^J (e^{-it_n^j\Delta}\psi^j)(\cdot-x_n^j)+R_n^J.$$
If $t_n^j\to\infty$, by Lemma \ref{linear scattering}, there exists $\tilde{\psi}^j\in H^1$ such that $\|e^{it_n^j\mathcal{H}}\tilde{\psi}^j-e^{-it_n^j\Delta}\psi^j\|_{H^1}\to 0$. Otherwise ($t_n^j=0$), we set $\tilde{\psi}^j=\psi^j$. Then, we write 
$$u_n=\sum_{j=1}^J e^{it_n^j\mathcal{H}}(\tilde{\psi}^j(\cdot-x_n^j))+\tilde{R}_n^J,$$
where 
$$\tilde{R}_n^J=R_n^J+\sum_{j=1}^Je^{-it_n^j\Delta}(\psi^j(\cdot-x_n^j))-e^{it_n^j\mathcal{H}}(\tilde{\psi}^j(\cdot-x_n^j)).$$
It suffices to show the asymptotic smallness $(5.6)$. Indeed, by the argument to prove $(5.12)$, one can prove that 
$$\lim_{J\to\infty}\Big[\lim_{n\to\infty}\|e^{-it\mathcal{H}}R_n^J\|_{S(\dot{H}^{1/2})}\Big]=0.$$
If $t_n^j=0$, it is obvious that 
$$e^{-it_n^j\Delta}(\psi^j(\cdot-x_n^j))-e^{it_n^j\mathcal{H}}(\tilde{\psi}^j(\cdot-x_n^j))=\psi^j-\psi^j=0.$$
If $t_n^j\to\infty$ and $x_n^j=0$, by the Sobolev inequality and Strichartz estimates, we get
\begin{align*}
&\|e^{-it\mathcal{H}}(e^{-it_n^j\Delta}\psi^j-e^{it_n^j\mathcal{H}}\tilde{\psi}^j)\|_{S(\dot{H}^{1/2})}\lesssim\|\mathcal{H}^{1/2}e^{-it\mathcal{H}}(e^{-it_n^j\Delta}\psi^j-e^{it_n^j\mathcal{H}}\tilde{\psi}^j)\|_{S(L^2)}\\
&\lesssim\|\mathcal{H}^{1/2}(e^{-it_n^j\Delta}\psi^j-e^{it_n^j\mathcal{H}}\tilde{\psi}^j)\|_{L^2}\sim\|e^{-it_n^j\Delta}\psi^j-e^{it_n^j\mathcal{H}}\tilde{\psi}^j\|_{\dot{H}^{1/2}}\to 0.
\end{align*}
If $x_n^j\to\infty$, by $(5.11)$, Kato's inhomogeneous Strichartz estimate and the argument used in $(5.12)$, we obtain
\begin{align*}
&\|e^{-it\mathcal{H}}(e^{it_n^j\mathcal{H}}(\psi^j(\cdot-x_n^j))-e^{-it_n^j\Delta}(\psi^j(\cdot-x_n^j)) )\|_{S(\dot{H}^{1/2})}\\
&=\Big\|e^{-it\mathcal{H}}\Big(\int_{-t_n^j}^0 e^{-i(t_n^j+s)\mathcal{H}}(Ve^{is\Delta}(\psi^j(\cdot-x_n^j)))\Big)\Big\|_{S(\dot{H}^{1/2})}\\
&\lesssim\|Ve^{-it\Delta}(\psi^j(\cdot-x_n^j))\|_{L_t^4L_x^{6/5}}\to 0.
\end{align*}
Collecting all, we conclude that $\tilde{R}_n^J$ has asymptotic smallness property.
\end{proof}

\begin{proof}[Proof of Proposition \ref{prop:ProfileDecomposition} when $r_n=1$, without the extra regularity assumption]
Repeating the argument in \cite{HR, DHR}, we obtain a profiles decomposition
\begin{equation}
u_n=\sum_{j=1}^J e^{it_n^j\mathcal{H}}(\psi^j(\cdot-x_n^j))+R_n^J
\end{equation}
with properties $(5.4)\sim(5.8)$. We omit the construction of this profile decomposition, since it is exactly the same as that in \cite{DHR} except that we need to use  norm equivalence in several steps.

It remains to modify the profile decomposition $(5.13)$ to obey $(5.2)$ and $(5.3)$. For each $j$, passing to a subsequence, we may assume that $t_n^j\to t_*^j\in\mathbb{R}\cup\{\infty\}$ and $x_n^j\to x_*^j\in\mathbb{R}^3\cup\{\infty\}$. If $t_*^j\neq\infty$ and $x_*^j\neq\infty$, we replace $e^{it_n^j\mathcal{H}}(\psi^j(\cdot-x_n^j))$ by $\tilde{\psi}^j=e^{it_*^j\mathcal{H}}(\psi^j(\cdot-x_*^j))$. If $t_*^j=\infty$ and $x_*^j\neq\infty$, we replace $e^{it_n^j\mathcal{H}}(\psi^j(\cdot-x_n^j))$ by $e^{it_n^j\mathcal{H}}\tilde{\psi}^j$, where $\tilde{\psi}=\psi^j(\cdot-x_*^j)$. If $t_*^j\neq\infty$ and $x_*^j=\infty$, we replace $e^{it_n^j\mathcal{H}}(\psi^j(\cdot-x_n^j))$ by $\tilde{\psi}^j(\cdot-x_n^j)$, where $\tilde{\psi}^j=e^{-it_*^j\Delta}\psi^j$. We claim that the remainder is still asymptotically small in the sense of $(5.6)$ (thus, we may assume that $t_n^j=0$ or $t_n^j\to\infty$, and $x_n^j=0$ or $x_n^j\to\infty$). Indeed, in the last case, we have
\begin{align*}
&\|e^{-it\mathcal{H}}(e^{it_n^j\mathcal{H}}(\psi^j(\cdot-x_n^j))-e^{-it_*^j\Delta}(\psi^j(\cdot-x_n^j)) )\|_{S(\dot{H}^{1/2})}\\
&=\|e^{-it\mathcal{H}}(e^{it_n^j\mathcal{H}}(\psi^j(\cdot-x_n^j))-e^{-it_n^j\Delta}(\psi^j(\cdot-x_n^j)) )\|_{S(\dot{H}^{1/2})}+o_n(1).
\end{align*}
Then, by estimates in $(5.12)$, we prove that 
$$\|e^{-it\mathcal{H}}(e^{it_n^j\mathcal{H}}(\psi^j(\cdot-x_n^j))-e^{-it_n^j\Delta}(\psi^j(\cdot-x_n^j)) )\|_{S(\dot{H}^{1/2})}\to 0.$$
By the same way, one can show that other modifications are harmless. Moreover, one can modify the profile decomposition to satisfy $(5.3)$.
\end{proof}

\begin{corollary}[Energy Pythagorean expansion]\label{prop:EPythagorean} In the situation of Proposition \ref{prop:ProfileDecomposition},
\begin{equation}\label{eq:EPythagorean}
E_{V_{r_n}}[u_n]=\sum_{j=1}^J E_{V_{r_n}}[e^{it_n^j\mathcal{H}_{r_n}}(\psi^j(\cdot-x_n^j))]+E_{V_{r_n}}[R_n^J]+o_n(1).
\end{equation}
\end{corollary}

\begin{proof}
By (5.8), it suffices to show that
\begin{equation}
\Big\|\sum_{j=1}^J e^{it_n^j\mathcal{H}_{r_n}}(\psi^j(\cdot-x_n^j))+R_n^J\Big\|_{L^4}^4=\sum_{j=1}^J\|e^{it_n^j\mathcal{H}_{r_n}}(\psi^j(\cdot-x_n^j))\|_{L^4}^4+\|R_n^J\|_{L^4}^4+o_n(1).
\end{equation}
For arbitrary small $\epsilon>0$, let $\psi_\epsilon^j\in C_c^\infty$ such that $\|\psi^j-\psi_\epsilon^j\|_{H^1}\leq\epsilon/J$. Replacing $\psi^j$ by $\psi_\epsilon^j$ in $(5.15)$ with $O(\epsilon)$-error, one may assume that $\psi^j\in C_c^\infty$. First, we observe that
$$\Big\|\sum_{j=1}^J e^{it_n^j\mathcal{H}_{r_n}}(\psi^j(\cdot-x_n^j))\Big\|_{L^4}^4=\sum_{j=1}^J\|e^{it_n^j\mathcal{H}_{r_n}}(\psi^j(\cdot-x_n^j))\|_{L^4}^4+o_n(1).$$
Indeed, each cross term of its left left hand side is of the form
\begin{equation}\label{eq:EPythagoreanCross}
\int_{\mathbb{R}^3} e^{it_n^{j_1}\mathcal{H}_{r_n}}(\psi^{j_1}(\cdot-x_n^{j_1}))\overline{e^{it_n^{j_2}\mathcal{H}_{r_n}}(\psi^{j_2}(\cdot-x_n^{j_2}))}e^{it_n^{j_3}\mathcal{H}_{r_n}}(\psi^{j_3}(\cdot-x_n^{j_3}))\overline{e^{it_n^{j_4}\mathcal{H}_{r_n}}(\psi^{j_4}(\cdot-x_n^{j_4}))}dx.
\end{equation}
If there is one $j_k$ such that $t_n^{j_k}\to \infty$, for example, say $t_n^{j_1}\to\infty$, by the dispersive estimate, the Sobolev inequality and the norm equivalence, we have
\begin{align*}
|(\ref{eq:EPythagoreanCross})|&\leq \|e^{it_n^{j_1}\mathcal{H}_{r_n}}(\psi^{j_1}(\cdot-x_n^{j_1}))\|_{L^4}\prod_{k=2,3,4}\|e^{it_n^{j_k}\mathcal{H}_{r_n}}(\psi^{j_k}(\cdot-x_n^{j_k}))\|_{L^4}\\
&\lesssim|t_n^{j_1}|^{-\frac{3}{4}}\|\psi^{j_1}\|_{L^{4/3}}\|\psi^{j_2}\|_{H^1}\|\psi^{j_3}\|_{H^1}\|\psi^{j_4}\|_{H^1}\to 0.
\end{align*}
Otherwise (all $t_n^j$ are zero), then $|x_n^{j_1}-x_n^{j_2}|\to\infty$. Thus (\ref{eq:EPythagoreanCross}) converges to zero as $n\to\infty$.

Moreover, we have
$$\lim_{J_1\to\infty}\underset{n\to\infty}\limsup\|R_n^{J_1}\|_{L^4}=0.$$
Indeed, by (\ref{eq:SmallError}) and (5.8),
\begin{align*}
&\|R_n^{J_1}\|_{L^4}\leq\|e^{-it\mathcal{H}_{r_n}}R_n^{J_1}\|_{L_t^\infty L_x^4}\leq\|e^{-it\mathcal{H}_{r_n}}R_n^{J_1}\|_{L_t^\infty L_x^3}^{1/2}\|e^{-it\mathcal{H}_{r_n}}R_n^{J_1}\|_{L_t^\infty L_x^6}^{1/2}\\
&\leq\|e^{-it\mathcal{H}_{r_n}}R_n^{J_1}\|_{S(\dot{H}^{1/2})}^{1/2}\|e^{-it\mathcal{H}_{r_n}}R_n^{J_1}\|_{\dot{H}^1}^{1/2}\leq\|e^{-it\mathcal{H}_{r_n}}R_n^{J_1}\|_{S(\dot{H}^{1/2})}^{1/2}\sup_n\|u_n\|_{H^1}^{1/2}.
\end{align*}
Thus, for $\epsilon>0$, there exists $J_1\gg1$ such that $\|R_n^{J_1}\|_{L^4}\leq\epsilon$ for all sufficiently large $n$. Hence, we obtain
\begin{align*}
\|u_n\|_{L^4}^4&=\sum_{j=1}^{J_1}\|e^{it_n^j\mathcal{H}_{r_n}}(\psi^j(\cdot-x_n^j))\|_{L^4}^4+O(\epsilon)+o_n(1)\\
&=\sum_{j=1}^J\|e^{it_n^j\mathcal{H}_{r_n}}(\psi^j(\cdot-x_n^j))\|_{L^4}^4+\|R_n^{J_1}-R_n^J\|_{L^4}^4+O(\epsilon)+o_n(1)\\
&=\sum_{j=1}^J\|e^{it_n^j\mathcal{H}_{r_n}}(\psi^j(\cdot-x_n^j))\|_{L^4}^4+\|R_n^J\|_{L^4}^4+O(\epsilon)+o_n(1).
\end{align*}
\end{proof}
\section{Construction of a Minimal Blow-up Solution}

We define the critical mass-energy $\mathcal{ME}_c$ by the supremum over all $\ell$ such that
\begin{equation}\label{ell}
M[u_0]E[u_0]<\ell,\ \|u_0\|_{L^2}\|\mathcal{H}^{1/2}u_0\|_{L^2}<\alpha\Rightarrow\|\textup{NLS}_V(t)u_0\|_{S(\dot{H}^{1/2})}<\infty.
\end{equation}
Here, $\mathcal{ME}_c$ is a strictly positive number. Indeed, by the Sobolev inequality, Strichartz estimates, the norm equivalence and comparability of gradient and energy (Proposition \ref{Comparability}), we have
\begin{align*}
\|e^{-it\mathcal{H}}u_0\|_{S(\dot{H}^{1/2})}^4&\lesssim\|\mathcal{H}^{1/4}e^{-it\mathcal{H}}u_0\|_{S(L^2)}^4\lesssim\|\mathcal{H}^{1/4}u_0\|_{L^2}^4\sim\||\nabla|^{1/2}u_0\|_{L^2}^4\\
&\leq\|u_0\|_{L^2}^2\|\nabla u_0\|_{L^2}^2\sim\|u_0\|_{L^2}^2\|\mathcal{H}^{1/2}u_0\|_{L^2}^2\sim M[u_0]E[u_0].
\end{align*}
Hence, it follows from the small data scattering (Corollary \ref{SmallDataScattering}) that $\eqref{ell}$ holds for all sufficiently small $\ell>0$. Note that the scattering conjecture (Conjecture \ref{Conjecture}) is false if and only if $\mathcal{ME}_c<\mathcal{ME}$.

In this section, assuming that the scattering conjecture fails, we construct a global solution having infinite Strichart norm $\|\cdot\|_{S(\dot{H}^{1/2})}$ at the critical mass-energy $\mathcal{ME}_c$. 

\begin{theorem}[Minimal blow-up]\label{thm:MinimalBlowup}
If Conjecture \ref{Conjecture} is false, there exists $u_{c,0}\in H^1$ such that
$$M[u_{c,0}]E[u_{c,0}]=\mathcal{ME}_c,\ \|u_{c,0}\|_{L^2}\|\mathcal{H}^{1/2}u_{c,0}\|_{L^2}<\alpha$$
and
$$\|u_c(t)\|_{S(\dot{H}^{1/2})}=\infty,$$
where $u_c(t)$ is the solution to $(\textup{NLS}_V)$ with initial data $u_{c,0}$.
\end{theorem}

\begin{proof}
By the assumption, there exists a sequence $\{u_{n,0}\}_{n=1}^\infty$ such that
$$M[u_{n,0}]E[u_{n,0}]\downarrow\mathcal{ME}_c,\ \|u_{n,0}\|_{L^2}\|\mathcal{H}^{1/2}u_{n,0}\|_{L^2}<\alpha$$
and
$$\|u_n(t)\|_{S(\dot{H}^{1/2})}=\infty,$$
where $u_n(t)=\textup{NLS}_V(t)u_{n,0}$. We will extract a critical element $u_{c,0}$ from the sequence $\{u_{n,0}\}_{n=1}^\infty$ by two steps.\\
\textbf{(Step 1. Boundedness of $\{u_{n,0}\}_{n=1}^\infty$)} We will show that $\{u_{n,0}\}_{n=1}^\infty$ is bounded in $H^1$. To this end, it suffices to show that passing to a subsequence, 
\begin{equation}
r_n=\|u_{n,0}\|_{L^2}^{-2}\sim 1,
\end{equation}
since by the norm equivalence,
$$\|u_{n,0}\|_{H^1}^2=\|u_{n,0}\|_{L^2}^2+\|\nabla u_{n,0}\|_{L^2}^2\sim\|u_{n,0}\|_{L^2}^2+\|\mathcal{H}^{1/2} u_{n,0}\|_{L^2}^2=r_n^{-2}+\alpha^2 r_n^2.$$
We assume that $r_n\to 0$ or $r_n\to+\infty$, and consider the scaled sequence 
$$\{\tilde{u}_{n}(t,x)\}_{n=1}^\infty=\{\tfrac{1}{r_n}u_n(\tfrac{t}{r_n^2}, \tfrac{x}{r_n})\}_{n=1}^\infty\textup{ and }\{\tilde{u}_{n,0}\}_{n=1}^\infty=\{\tfrac{1}{r_n}u_n(\tfrac{\cdot}{r_n})\}_{n=1}^\infty.$$
Then, each $\tilde{u}_n$ solves 
\begin{equation}\tag{$\textup{NLS}_{V_{r_n}}$}
i\partial_t \tilde{u}_n-\mathcal{H}_{r_n}\tilde{u}_n+|\tilde{u}_n|^2\tilde{u}_n=0,\ \tilde{u}_n(0)=\tilde{u}_{n,0}.
\end{equation}
The goal is now to show that $\|\tilde{u}_n\|_{S(\dot{H}^{1/2})}=\|u_n\|_{S(\dot{H}^{1/2})}=\infty$ for sufficiently large $n$, which contradicts to the choice of $\{u_{n,0}\}_{n=1}^\infty$. To this end, we construct an approximation $w_n^J(t)$ of $\tilde{u}_n(t)$, and then we show that $\|w_n^J\|_{S(\dot{H}^{1/2})}=\infty$ for sufficiently large $n$. Finally, comparing $\tilde{u}_n(t)$ with $w_n^J(t)$ by the long time perturbation lemma, we prove that $\tilde{u}_n(t)$ also has infinite Strichartz norm $\|\cdot\|_{S(\dot{H}{1/2})}$.\\

Note that $\{\tilde{u}_{n,0}\}_{n=1}^\infty$ is bounded in $H^1$, since $\|\tilde{u}_{n,0}\|_{L^2}^2=r_n\|u_{n,0}\|_{L^2}^2=1$ and
$$\|\nabla \tilde{u}_{n,0}\|_{L^2}^2\sim\|\mathcal{H}_{r_n}^{1/2}\tilde{u}_{n,0}\|_{L^2}^2=\|u_{n,0}\|_{L^2}^2\|\mathcal{H}^{1/2}u_{n,0}\|_{L^2}^2<\alpha^2.$$
Therefore, by Proposition \ref{prop:ProfileDecomposition}, extracting to a subsequence, we have
$$\tilde{u}_{n,0}=\sum_{j=1}^J e^{it_n^j\mathcal{H}_{r_n}}(\psi^j(\cdot-x_n^j))+R_n^J.$$
For each $j$, if $t_n^j\to\infty$, by Proposition \ref{prop:WaveOperator} (with $V=0$), we get $\tilde{\psi}^j\in H^1$ such that
\begin{equation}
\|\textup{NLS}_0(-t_n^j)\tilde{\psi}^j-e^{-it_n^j\Delta}\psi^j\|_{H^1}\to0.
\end{equation}
If $t_n^j=0$, we set $\tilde{\psi}^j=\psi^j$. Replacing each linear profile by the nonlinear profile, we define the approximation of $\tilde{u}_n(t)$ by
$$w_n^J(t,x)=\sum_{j=1}^Jv^j(t-t_n^j, x-x_n^j),$$
where
$$v^j(t,x)=\textup{NLS}_0(t)\tilde{\psi}^j.$$

Let $\tilde{w}_n^J(t)=\textup{NLS}_0(t)w_n^J(0)$. We will show that there exists $A_0>0$, independent of $J$, such that
\begin{equation}
\|\tilde{w}_n^J(t)\|_{S(\dot{H}^{1/2})}\leq A_0
\end{equation}
for all $n\geq n_0=n_0(J)$. Indeed, we have
\begin{align*}
E_0[w_n^J(0)]&=\sum_{j=1}^J E_0[v^j(-t_n^j, \cdot-x_n^j)]+o_n(1)&\textup{ (by orthogonality of $(t_n^j, x_n^j)$)}\\
&=\sum_{j=1}^JE_0[e^{-it_n^j\Delta}\psi^j(\cdot-x_n^j)]+o_n(1)&\textup{ (by $(6.3)$ when $t_n^j\to\infty$)}\\
&=\sum_{j=1}^J E_{V_{r_n}}[e^{it_n^j\mathcal{H}_{r_n}}(\psi^j(\cdot-x_n^j))]+o_n(1)&\textup{ (by the argument in $(5.12)$)}\\
&\leq E_{V_{r_n}}[\tilde{u}_{n,0}]+o_n(1)=r_n^{-1} E_V[u_{n,0}]+o_n(1)&\textup{ (by Corollary \ref{prop:EPythagorean})}.
\end{align*}
Similarly, one can show that 
\begin{align*}
M[w_n^J(0)]&\leq M[\tilde{u}_{n,0}]=r_nM[u_{n,0}]+o_n(1),\\
\|\nabla w_n^J(0)\|_{L^2}&\leq \|\mathcal{H}_{r_n}^{1/2}\tilde{u}_{n,0}\|=r_n^{-1}\|\mathcal{H}^{1/2}u_{n,0}\|_{L^2}+o_n(1).
\end{align*}
Therefore, we obtain that 
\begin{equation}
\begin{aligned}
M[w_n^J(0)]E_0[w_n^J(0)]&\leq M[u_{n,0}]E[u_{n,0}]+o_n(1)=\mathcal{ME}_c+o_n(1)<\mathcal{ME}\\
\|w_n^J(0)\|_{L^2}\|\nabla w_n^J(0)\|_{L^2}&\leq\|u_{n,0}\|_{L^2}\|\mathcal{H}^{1/2}u_{n,0}\|_{L^2}+o_n(1)<\alpha.
\end{aligned}
\end{equation}
Moreover, we have
\begin{equation}
\mathcal{ME}\leq M[Q]E_0[Q]\textup{ and }\alpha\leq \|Q\|_{L^2}\|\nabla Q\|_{L^2}.
\end{equation}
Indeed, if $V\geq0$, $(6.6)$ is trivial. If $V$ has a nontrivial negative part, by the Gagliardo-Nirenberg inequality and the Pohozaev identities, 
\begin{align*}
\frac{4}{3\sqrt{3}\|\mathcal{Q}\|_{L^2}^3}&=\frac{\|\mathcal{Q}\|_{L^4}^4}{\|\mathcal{Q}\|_{L^2}\|\mathcal{H}^{1/2}\mathcal{Q}\|_{L^2}^3}\\
&\geq \lim_{n\to\infty}\frac{\|Q(\cdot-n)\|_{L^4}^4}{\|Q(\cdot-n)\|_{L^2}\|\mathcal{H}^{1/2}Q(\cdot-n)\|_{L^2}^3}=\frac{\|Q\|_{L^4}^4}{\|Q\|_{L^2}\|\nabla Q\|_{L^2}^3}=\frac{4}{3\sqrt{3}\|Q\|_{L^2}^3}.
\end{align*}
Thus, by the Pohozaev identities again, we obtain that
$$\mathcal{ME}=M[\mathcal{Q}]E[\mathcal{Q}]=\frac{1}{2}\|\mathcal{Q}\|_{L^2}^4\leq \frac{1}{2}\|Q\|_{L^2}^4=M[Q]E_0[Q]$$
and
$$\alpha=\|\mathcal{Q}\|_{L^2}\|\mathcal{H}^{1/2}\mathcal{Q}\|_{L^2}=\sqrt{3}\|\mathcal{Q}\|_{L^2}^2<\sqrt{3}\|Q\|_{L^2}^2=\|Q\|_{L^2}\|\nabla Q\|_{L^2}.$$
Combining $(6.5)$ and $(6.6)$, we prove that
$$M[w_n^J(0)]E_0[w_n^J(0)]<M[Q]E_0[Q]\textup{ and }\|w_n^J(0)\|_{L^2}\|\nabla w_n^J(0)\|_{L^2}<\|Q\|_{L^2}\|\nabla Q\|_{L^2}.$$
Therefore, $(6.4)$ follows from the scattering theorem for the homogeneous nonlinear Schr\"odinger equation \cite{HR, DHR}.

Next, we claim that there exists $A_1>0$, independent of $J$, such that
\begin{equation}
\|w_n^J(t)\|_{S(\dot{H}^{1/2})}\leq A_1
\end{equation}
for all $n\geq n_1=n_1(J)$. To see this, we observe that $w_n^J(t)$ solves 
$$i\partial_t w_n^J+\Delta w_n^J+|w_n^J|^2w_n^J=e,$$
where
$$e=|w_n^J|^2w_n^J-\sum_{j=1}^J|v^j(t-t_n^j, x-x_n^j)|^2v^j(t-t_n^j, x-x_n^j).$$
Here, by the asymptotic orthogonality of parameters $(t_n^j, x_n^j)$, the cross terms in $e$ vanishes as $n\to\infty$. Hence, we have
$$\|e\|_{S'(\dot{H}^{-1/2})}\leq\epsilon_0$$
for all sufficiently large $n$, where $\epsilon_0=\epsilon_0(A_0)$ is given by Lemma 2.13 with $V=0$. Therefore, $(6.7)$ follows from Lemma 2.13.

Finally, we deduce a contradiction using $(6.7)$. We observe that $w_n^J(t)$ satisfies 
$$i\partial_t w_n^J-\mathcal{H}_{r_n}w_n^J+|w_n^J|^2w_n^J=e_n^J,$$
where
$$e_n^J=-V_{r_n}w_n^J+|w_n^J|^2w_n^J-\sum_{j=1}^J |v^j(\cdot-t_n^j, \cdot-x_n^j)|^2 v^j(\cdot-t_n^j, \cdot-x_n^j).$$
Let $\epsilon_0'=\epsilon_0'(A_1)$ be a small number given in the long time perturbation lemma. We claim that there exists $J\gg1$ such that
\begin{align}
\|e^{it\mathcal{H}_{r_n}}(\tilde{u}_{n,0}-w_n^J(0))\|_{S(\dot{H}^{1/2})}<\epsilon_0',\\
\|e_n^J\|_{S'(\dot{H}^{-1/2})}<\epsilon_0'
\end{align}
for all $n\geq n_3=n_3(J)\gg1$. For $(6.8)$, we write
$$\tilde{u}_{n,0}-w_n^J(0)=R_n^J+\sum_{j=1}^J\Big(e^{it_n^j\mathcal{H}_{r_n}}(\psi^j(\cdot-x_n^j))-v^j(-t_n^j, \cdot-x_n^j)\Big).$$
By Proposition \ref{prop:ProfileDecomposition}, one can choose $J\gg1$ such that 
$$\|e^{it\mathcal{H}_{r_n}}R_n^J\|_{S(\dot{H}^{1/2})}<\tfrac{\epsilon_0'}{2}$$
for all $n\geq n_3$. Hence, it suffices to show that for each $j$,
$$\|e^{it\mathcal{H}_{r_n}}(e^{it_n^j\mathcal{H}_{r_n}}(\psi^j(\cdot-x_n^j))-v^j(-t_n^j, \cdot-x_n^j))\|_{S(\dot{H}^{1/2})}\to 0.$$
Indeed, arguing as in $(5.12)$, one can show that
$$\Big\|e^{it\mathcal{H}_{r_n}}\Big(e^{it_n^j\mathcal{H}_{r_n}}(\psi^j(\cdot-x_n^j))-e^{-it_n^j\Delta}\psi^j(\cdot-x_n^j)\Big)\Big\|_{S(\dot{H}^{1/2})}\to 0.$$
Moreover, by the Sobolev inequality, Strichartz estimates and the choice of $\tilde{\psi}^j$,
$$\|e^{it\mathcal{H}_{r_n}} (e^{-it_n^j\Delta}\psi^j(\cdot-x_n^j)-v^j(-t_n^j, \cdot-x_n^j))\|_{S(\dot{H}^{1/2})}\to0.$$
It is easy to check $(6.9)$, since $r_n\to 0$ or $r_n\to +\infty$ and $v^j(\cdot-t_n^j, \cdot-x_n^j)$'s are asymptotically orthogonal each other. Finally, applying the long time perturbation lemma to $\tilde{u}_n(t)$ and $w_n^J(t)$ with $(6.7)$, $(6.8)$ and $(6.9)$, we conclude that $\|\tilde{u}_n(t)\|_{S(\dot{H}^{1/2})}<\infty$ for large $n$.\\
\textbf{(Step 2. Extraction of a critical element)} 
Now, we extract $u_{c,0}$ from a bound sequence $\{u_{n,0}\}_{n=1}^\infty$. We only sketchy this step, because it is similar to the proof of \cite[Proposition 5.4]{HR}. Indeed, it suffices to replace the linear profile $e^{it\Delta}$ by $e^{-it\mathcal{H}}=e^{it(\Delta-V)}$ in the proof. First, by the argument in \cite{DHR, HR} (but using Proposition \ref{prop:ProfileDecomposition} with $r_n=1$), one can show that passing to a subsequence, $(u_{n,0})$ has only one nonlinear profile 
$$u_{n,0}=e^{it_n^1\mathcal{H}}(\psi^1(\cdot-x_n^1))+R_n^1.$$
If $x_n^1\to \infty$, let $v^1(t)=\textup{NLS}_0(t)\psi^1$. Comparing $u_n$ with $v^1(\cdot-t_n^1,\cdot-x_n^1))$, one can deduce a contradiction as in Step 1. Hence, $x_n^1=0$. If $t_n^1\to\infty$, by Proposition 4.3, we pick $\tilde{\psi}^1$ such that $\|e^{it_n^1\mathcal{H}}\psi^1-\textup{NLS}(-t_n^1)\tilde{\psi}^1\|_{H^1}\to 0$. If $t_n^1=0$, let $\tilde{\psi}^1=\psi^1$. We set $u_{c,0}=\tilde{\psi}^1$. Then, by the argument in \cite{HR}, one can show that $u_{c,0}$ satisfies the desired properties in Theorem 6.1. 
\end{proof}

\begin{proposition}[Precompactness of a minimal blow-up solution]\label{prop:Precompactness}
Let $u_c(t)$ be in Theorem \ref{thm:MinimalBlowup}. Then $K:=\{u_c(t):t\in \mathbb{R}\}$ is precompact in $H^1$.
\end{proposition}

\begin{proof}
Let $\{t_n\}_{m=1}^\infty$ be a sequence in $\mathbb{R}$. Passing to a subsequence, we may assume that $t_n\to t_*\in[-\infty,+\infty]$. If $t_*\neq\infty$, then $u_c(t_n)\to u_c(t_*)$ in $H^1$. Suppose that $t_*=\infty$. Applying Proposition \ref{prop:ProfileDecomposition} to $\{u_c(t_n)\}_{n=1}^\infty$, we write
$$u_c(t_n)=\sum_{j=1}^J e^{it_n^j\mathcal{H}}(\psi^j(\cdot-x_n^j)) +R_n^J.$$
If $\psi^j\neq 0$ for some $j\geq 2$ by the argument in the proof of \cite[Proposition 5.5]{HR}, one can deduce a contradiction. Therefore, we have
$$u_c(t_n)=e^{it_n^1\mathcal{H}}(\psi^1(\cdot-x_n^1))+R_n^1.$$
If $x_n^1\to\infty$, approximating $e^{it_n^1\mathcal{H}}(\psi^1(\cdot-x_n^1))=(e^{it_n^1(-\Delta+V(\cdot+x_n^1))}\psi^1)(\cdot-x_n^1)$ by the nonlinear profile $(\textup{NLS}_0(-t_n^1))\psi^1(\cdot-x_n^1)$ as in the proof of Theorem \ref{thm:MinimalBlowup}, one can deduce a contradiction from the homogeneous case (Theorem 1.7). Hence, $x_n^1=0$. It remains to show $R_n^1\to 0$ in $H^1$ and $t_n^1=0$. The proof is very close to that of \cite[Proposition 5.5]{HR}, so we omit the proof.
\end{proof}

\begin{lemma}[Precompactness implies uniform localization]\label{lem:UniformLocalization}
Suppose that $K:=\{u(t): t\in\mathbb{R}\}$ is precompact in $H^1$. Then, for any $\epsilon>0$, there exists $R=R(\epsilon)>1$ such that
$$\sup_{t\in\mathbb{R}}\int_{|x|\geq R}|\nabla u(t,x)|^2+|u(t,x)|^2+|u(t,x)|^4dx\leq \epsilon.$$
\end{lemma}

\begin{proof}
The proof follows from exactly the same argument in \cite{HR}, so we omit it.
\end{proof}

\section{Extinction of a Minimal Blow-up Solution}

Finally, we prove Theorem \ref{thm:Scattering} eliminating a minimal blow-up solution via the localized vial identities.

\begin{proposition}[Localized virial identities]\label{prop:Virial} Let $\chi\in C_c^\infty(\mathbb{R}^3)$. Suppose that $u(t)$ is a solution to $(\textup{NLS}_V)$. Then,
\begin{align}
\partial_t\int_{\mathbb{R}^3}\chi|u|^2dx&=2\Im\int_{\mathbb{R}^3}(\nabla\chi\cdot\nabla u)\bar{u}dx,\\
\partial_t^2\int_{\mathbb{R}^3}\chi|u|^2dx&=4\sum_{i,j=1}^3\Re\int_{\mathbb{R}^3}\partial_{x_ix_j}\chi\partial_{x_i}u\overline{\partial_{x_j}u}dx-\int_{\mathbb{R}^3} \Delta\chi|u|^4dx\\
&\ \ \ \ -\int_{\mathbb{R}^3}\Delta^2\chi|u|^2 dx-2\int_{\mathbb{R}^3}(\nabla\chi\cdot\nabla V)|u|^2 dx,\nonumber.
\end{align}
\end{proposition}

\begin{proof} By the equation and by integration by parts, we get
\begin{align*}
\partial_t\int_{\mathbb{R}^3}\chi|u|^2dx&=2\Re\int_{\mathbb{R}^3}\chi\bar{u}\partial_t udx=-2\Im\int_{\mathbb{R}^3}\chi\bar{u}(\Delta u-Vu+|u|^2u)dx\\
&=-2\Im\int_{\mathbb{R}^3}\chi\bar{u}\Delta udx=2\Im\int_{\mathbb{R}^3}(\nabla\chi\cdot\nabla u)\bar{u}+\chi|\nabla u|^2dx\\
&=2\Im\int_{\mathbb{R}^3}(\nabla\chi\cdot\nabla u)\bar{u}dx.
\end{align*}
Differentiating $(7.1)$, we obtain that 
\begin{align*}
\partial_t^2\int_{\mathbb{R}^3}\chi|u|^2dx&=2\Im\int_{\mathbb{R}^3}(\nabla\chi\cdot\nabla \partial_t u)\bar{u}dx+2\Im\int_{\mathbb{R}^3}(\nabla\chi\cdot\nabla u)\overline{\partial_t u}dx\\
&=-2\Im\int_{\mathbb{R}^3}\Delta\chi \partial_t u\bar{u}dx+4\Im\int_{\mathbb{R}^3}(\nabla\chi\cdot\nabla u)\overline{\partial_t u}dx\\
&=-2\Re\int_{\mathbb{R}^3}\Delta \chi \Delta u\bar{u}dx+2\int_{\mathbb{R}^3}\Delta \chi V|u|^2dx-2\int_{\mathbb{R}^3}\Delta \chi |u|^4dx\\
&\ \ \ -4\Re\int_{\mathbb{R}^3}(\nabla\chi\cdot\nabla u) (\Delta\bar{u}-V\bar{u}+|u|^2\bar{u})dx.
\end{align*}
But, we have
\begin{align*}
2\Re\int_{\mathbb{R}^3}\Delta \chi \Delta u\bar{u}dx&=-2\Re\int_{\mathbb{R}^3}(\nabla\Delta\chi\cdot\nabla u)\bar{u}dx-2\int_{\mathbb{R}^3}\Delta\chi|\nabla u|^2dx\\
&=-\int_{\mathbb{R}^3}\nabla\Delta\chi\cdot \nabla(|u|^2)dx-2\int_{\mathbb{R}^3}\Delta\chi|\nabla u|^2dx\\
&=\int_{\mathbb{R}^3}\Delta^2\chi|u|^2dx-2\int_{\mathbb{R}^3}\Delta\chi|\nabla u|^2dx
\end{align*}
and
\begin{align*}
&4\Re\int_{\mathbb{R}^3}(\nabla\chi\cdot\nabla u) (\Delta\bar{u}-V\bar{u}+|u|^2\bar{u})dx\\
&=-4\Re\sum_{i,j=1}^3\int_{\mathbb{R}^3}\partial_{x_ix_j}\chi\partial_{x_i} u\overline{\partial_{x_j}u}dx-4\Re\sum_{i,j=1}^3\int_{\mathbb{R}^3}\partial_{x_i}\chi\partial_{x_ix_j} u\overline{\partial_{x_j}u}dx\\
&\ \ \ -2\int_{\mathbb{R}^3}V \nabla\chi\cdot\nabla(|u|^2)dx+\int_{\mathbb{R}^3}\nabla\chi\cdot\nabla(|u|^4)dx\\
&=-4\Re\sum_{i,j=1}^3\int_{\mathbb{R}^3}\partial_{x_ix_j}\chi\partial_{x_i} u\overline{\partial_{x_j}u}dx-2\sum_{i,j=1}^3\int_{\mathbb{R}^3}\partial_{x_i}\chi\partial_{x_i}(|\partial_{x_j}u|^2)dx\\
&\ \ \ +2\int_{\mathbb{R}^3}(\nabla\chi\cdot\nabla V)|u|^2dx+2\int_{\mathbb{R}^3}\Delta\chi V |u|^2dx-\int_{\mathbb{R}^3}\Delta\chi|u|^4dx\\
&=-4\Re\sum_{i,j=1}^3\int_{\mathbb{R}^3}\partial_{x_ix_j}\chi\partial_{x_i} u\overline{\partial_{x_j}u}dx+2\int_{\mathbb{R}^3}\Delta\chi|\nabla u|^2dx\\
&\ \ \ +2\int_{\mathbb{R}^3}(\nabla\chi\cdot\nabla V)|u|^2dx+2\int_{\mathbb{R}^3}\Delta\chi V |u|^2dx-\int_{\mathbb{R}^3}\Delta\chi|u|^4dx
\end{align*}
Therefore, we obtain $(7.2)$.
\end{proof}

\begin{proof}[Proof of Theorem \ref{thm:Scattering}]
If Conjecture \ref{Conjecture} fails, there exists a minimal blow-up solution $u_c(t)$ in Theorem \ref{thm:MinimalBlowup}. Choose a radially symmetric function $\chi\in C_c^\infty$ such that $\chi(x)=|x|^2$ for $|x|\leq 1$ and $\chi(x)=0$ for $|x|\geq 2$, and define
$$z_R(t):=\int_{\mathbb{R}^3}\chi_R|u_c(t)|^2dx$$
where $R>0$ and $\chi_R:=R^2\chi(\frac{\cdot}{R})$. Because $V$ is positive, by (7.1) and Theorem 1.4 $(i)$, we have
\begin{equation}\label{eq:z'Estimate}
\begin{aligned}
|z_R'(t)|&\leq\int_{\mathbb{R}^3} |\nabla\chi_R||u_c(t)||\nabla u_c(t)| dx\leq R\|u_c(t)\|_{L^2}\|\nabla u_c(t)\|_{L^2}\\
&\leq R\|u_{c,0}\|_{L^2}\|\mathcal{H}^{1/2}u_c(t)\|_{L^2}<R\alpha.
\end{aligned}
\end{equation}
On the other hand, by (7.2), we have
$$z_R''(t)=8\|\nabla u_c(t)\|_{L^2}^2-6\|u_c(t)\|_{L^4}^4-4\int_{\mathbb{R}^3}(x\cdot\nabla V)|u_c(t)|^2 dx+\textup{(remainder)},$$
where
\begin{align*}
\textup{(remainder)}&=4\sum_{i,j=1}^3\Re\int_{R\leq |x|\leq 2R}\partial_{x_ix_j}\chi_R\partial_{x_i}u_c(t)\overline{\partial_{x_j}u_c(t)}dx-8\|\nabla u_c(t)\|_{L^2(|x|\geq 2R)}^2\\
&-\int_{R\leq |x|\leq 2R} \Delta\chi_R|u_c(t)|^4dx+6\|u_c(t)\|_{L^4(|x|\geq 2R)}^4-\int_{\mathbb{R}^3}\Delta^2\chi_R|u_c(t)|^2 dx\\
&-2\int_{R\leq |x|\leq 2R}(\nabla\chi_R\cdot V)|u_c(t)|^2 dx+4\int_{|x|\geq 2R}(x\cdot\nabla V)|u_c(t)|^2 dx.
\end{align*}
We claim that there exists a constant $c_0>0$, independent of $R$, such that
\begin{equation}
8\|\nabla u_c(t)\|_{L^2}^2-6\|u_c(t)\|_{L^4}^4-4\int_{\mathbb{R}^3}(x\cdot\nabla V)|u_c(t)|^2 dx\geq c_0>0.
\end{equation}
Indeed, by the Pohozaev identities, we have
$$E_0[Q]=\frac{1}{2}\|\nabla Q\|_{L^2}^2-\frac{1}{4}\|Q\|_{L^4}^4=\frac{1}{2}\|Q\|_{L^2}^2,$$
and thus
$$\frac{\|Q\|_{L^4}^4}{\|Q\|_{L^2}\|\nabla Q\|_{L^2}^3}=\frac{4\|Q\|_{L^2}^2}{\|Q\|_{L^2}3\sqrt{3}\|Q\|_{L^2}^3}=\frac{4}{3\sqrt{3}\|Q\|_{L^2}^2}=\frac{4}{3\sqrt{6}M[Q]^{1/2}E_0[Q]^{1/2}}.$$
Moreover, since $V$ is positive, by Lemma 4.1, we have
$$\|\nabla u_c(t)\|_{L^2}^2\leq\|\mathcal{H}^{1/2}u_c(t)\|_{L^2}^2\leq 6E_V[u_{c,0}].$$
Therefore, using the ``free" Gagliardo-Nirenberg inequality, we obtain
\begin{align*}
\|u_c(t)\|_{L^4}^4&\leq\frac{\|Q\|_{L^4}^4}{\|Q\|_{L^2}\|\nabla Q\|_{L^2}^3}\|u_c(t)\|_{L^2}\|\nabla u_c(t)\|_{L^2}^3\\
&=\frac{4}{3\sqrt{6}M[Q]^{1/2}E_0[Q]^{1/2}}\|u_{c,0}\|_{L^2}\|\nabla u_c(t)\|_{L^2}^3\\
&\leq\frac{4}{3}\Big(\frac{M[u_{c,0}]E_V[u_{c,0}]}{M[Q]E_0[Q]}\Big)^{1/2}\|\nabla u_c(t)\|_{L^2}^2\\
&=\frac{4}{3}\Big(\frac{\mathcal{ME}_c}{\mathcal{ME}}\Big)^{1/2}\|\nabla u_c(t)\|_{L^2}^2.
\end{align*}
Then, it follows from replusivity of the potential, the norm equivalence and Lemma 4.1 that the left hand side of $(7.4)$ is greater than or equal to
\begin{align*}
8\|\nabla u_c(t)\|_{L^2}^2-6\|u_c(t)\|_{L^4}^4&\geq8\Big(1-\Big(\frac{\mathcal{ME}_c}{\mathcal{ME}}\Big)^{1/2}\Big)\|\nabla u_c(t)\|_{L^2}^2\sim\|\mathcal{H}^{1/2} u_c(t)\|_{L^2}^2\sim E[u_{c,0}].
\end{align*}
Next, we claim that 
\begin{align}
(\textup{remainder})\to 0\textup{ as }R\to\infty.
\end{align}
Indeed, the uniform localization of $u_c(t)$ (Lemma \ref{lem:UniformLocalization}) implies that
\begin{align*}
\textup{(remainder)}&\lesssim \|\nabla u_c(t)\|_{L^2(|x|\geq R)}^2+\|u_c(t)\|_{L^4(|x|\geq R)}^4+\frac{1}{R^2}\|u_c(t)\|_{L^2}^2\\
&\ \ \ +\|x\cdot\nabla V\|_{L^{3/2}}\|u_c(t)\|_{L^6(|x|\geq 2R)}^2\to 0.
\end{align*}
Combining (7.4) and (7.5), we obtain that
$$z_R''(t)\geq \frac{c_0}{2}$$
for sufficiently large $R>0$. Thus, $z_R'(t)\to+\infty$ as $t\to+\infty$, which contradicts to (\ref{eq:z'Estimate}).
\end{proof}

\appendix

\section{Positivity of the Schr\"odinger Operator}

The Schr\"odinger operator $\mathcal{H}$ is positive definite when the negative part of a potential is small.
\begin{lemma}[Positivity] If $V\in\mathcal{K}$, then
\begin{equation}\label{Positivity}
\int_{\mathbb{R}^3}|V||u|^2dx\leq\frac{\|V\|_{\mathcal{K}}}{4\pi}\|\nabla u\|_{L^2}^2.
\end{equation}
In particular, if $\|V_-\|_{\mathcal{K}}<4\pi$, then
$$\Big(1-\frac{\|V_-\|_{\mathcal{K}}}{4\pi}\Big)\|\nabla u\|_{L^2}^2\leq \|\mathcal{H}^{1/2}u\|_{L^2}^2=\int_{\mathbb{R}^3}\mathcal{H}u\overline{u}dx\leq \Big(1+\frac{\|V\|_{\mathcal{K}}}{4\pi}\Big)\|\nabla u\|_{L^2}^2.$$
\end{lemma}

\begin{proof}
Observe that 
\begin{align*}
\||V|^{1/2}(-\Delta)^{-1}|V|^{1/2}u\|_{L^2}^2&=\int_{\mathbb{R}^3}|V(x)|\Big|\int_{\mathbb{R}^3}\frac{|V(y)|^{1/2}}{4\pi|x-y|} u(y)dy\Big|^2 dx\\
&\leq\int_{\mathbb{R}^3}|V(x)|\Big(\int\frac{|V(y)|}{4\pi|x-y|}dy\Big)\int_{\mathbb{R}^3}\frac{|u(y)|^2}{4\pi|x-y|}dy dx\\
&\leq\Big(\frac{|V|_{\mathcal{K}}}{4\pi}\Big)\int_{\mathbb{R}^3}\int_{\mathbb{R}^3}\frac{|V(x)|}{4\pi|x-y|}|u(y)|^2dy dx\\
&\leq\Big(\frac{|V|_{\mathcal{K}}}{4\pi}\Big)^2\|u\|_{L^2}^2.
\end{align*}
Then, \eqref{Positivity} follows by the standard $TT^*$ argument with $T=|V|^{1/2}|\nabla|^{-1}$.
\end{proof}

\section{3d Cubic Defocusing NLS with a Potential}
In this section, we prove scattering for a 3d cubic defocusing NLS with a potential.

\begin{theorem}[Scattering for a cubic defocusing NLS with a potential]\label{defocusing}
Suppose that $V$ satisfies $(1.1)$ and $(1.2)$. We further assume that $\|(x\cdot\nabla V)_+\|_{\mathcal{K}}<4\pi$. Then, if $u(t)$ solves
\begin{equation}\label{defocusingNLS}
i\partial_t u+\Delta u-Vu-|u|^2u=0,\ u(0)=u_0\in H^1,
\end{equation}
then $u(t)$ scatters in $H^1$.
\end{theorem}

\begin{proof}
We only sketch the proof, since it follows by small modifications of the proof of Theorem \ref{thm:Scattering}. First, we claim that every $H^1$ solution to \eqref{defocusingNLS} is a global solution. Indeed, the $H^1$ norm of the solution $u(t)$ is controlled by the mass conservation law
$$M[u(t)]=\int_{\mathbb{R}^3} |u(t)|^2 dx=M[u_0]$$
and the energy conservation law
$$E[u(t)]=\frac{1}{2}\int_{\mathbb{R}^3}|\nabla u(t)|^2 +V|u(t)|^2 dx+\frac{1}{4}\int_{\mathbb{R}^3}\|u(t)|^4 dx=E[u_0].$$
In particular, by the smallness assumption on $V_-$, we have
$$\Big(1-\frac{\|V_-\|_{\mathcal{K}}}{4\pi}\Big)\|\nabla u(t)\|_{L^2}^2\leq \|\mathcal{H}^{1/2}u(t)\|_{L^2}^2\leq E[u(t)]=E[u_0].$$

Suppose that there is a solution having infinite $S(\dot{H}^{1/2})$ norm. Then, repeating the proof of Theorem \ref{thm:Scattering}, one can show that there is a critical element $u_c(t)$ that satisfies the uniform localization property in Lemma 6.3. Let $z_R(t)$ be as in the proof of Theorem \ref{thm:Scattering}. Then, by the virial identities for \eqref{defocusingNLS}
\begin{align*}
\partial_t\int_{\mathbb{R}^3}\chi|u|^2dx&=2\Im\int_{\mathbb{R}^3}(\nabla\chi\cdot\nabla u)\bar{u}dx,\\
\partial_t^2\int_{\mathbb{R}^3}\chi|u|^2dx&=4\sum_{i,j=1}^3\Re\int_{\mathbb{R}^3}\partial_{x_ix_j}\chi\partial_{x_i}u\overline{\partial_{x_j}u}dx+\int_{\mathbb{R}^3} \Delta\chi|u|^4dx\\
&\ \ \ \ -\int_{\mathbb{R}^3}\Delta^2\chi|u|^2 dx-2\int_{\mathbb{R}^3}(\nabla\chi\cdot\nabla V)|u|^2 dx,
\end{align*}
we obtain that 
\begin{equation}
|z_R'(t)|\leq R\|u_{c,0}\|_{L^2}\|\mathcal{H}^{1/2}u_c(t)\|_{L^2}\leq M[u_{c,0}]^{1/2}E[u_{c,0}]^{1/2}.
\end{equation}
Moreover, by $\eqref{Positivity}$, we have
\begin{align*}
|z_R''(t)|&\geq 8\|\nabla u_c(t)\|_{L^2}^2+6\|u_c(t)\|_{L^4}^4-4\int_{\mathbb{R}^3}(x\cdot\nabla V)|u_c(t)|^2 dx+o_R(1)\\
&\geq 4\Big(2-\frac{\|(x\cdot \nabla V)_+\|_{\mathcal{K}}}{4\pi}\Big)\|\nabla u_c(t)\|_{L^2}^2+6\|u_c(t)\|_{L^4}^4+o_R(1)\\
&\geq \beta\|\mathcal{H}^{1/2} u_c(t)\|_{L^2}^2+6\|u_c(t)\|_{L^4}^4+o_R(1),
\end{align*}
where
$$\beta=4\Big(2-\frac{\|(x\cdot \nabla V)_+\|_{\mathcal{K}}}{4\pi}\Big)\Big(1+\frac{\|V_+\|_{\mathcal{K}}}{4\pi}\Big)^{-1}.$$
By the assumption, $\beta$ is positive. If $\beta\geq 12$, then 
$$|z_R''(t)|\geq \min(24,2\beta) E[u_c(t)]+o_R(1)= \min(24,2\beta)E[u_{c,0}]+o_R(1).$$
We pick $R\gg1$ so that $|z_R''(t)|\geq c_0$ for all $t$. Thus, we have $|z_R'(t)|\to\infty$ as $t\to\infty$, which contradicts to (B.2).
\end{proof}

\end{document}